\documentclass[pdflatex,sn-mathphys-num]{sn-jnl}
\usepackage{graphicx}
\usepackage{amsmath,amssymb,amsthm,amsfonts,bm,mathtools}
\usepackage[linesnumbered,ruled]{algorithm2e}
\usepackage{subcaption}
\usepackage{float}
\usepackage{xcolor,textcomp,array}
\usepackage{pifont,makecell,diagbox,stmaryrd}
\usepackage{booktabs,threeparttable,multirow,adjustbox}
\usepackage{hhline}

\theoremstyle{thmstyleone}
\newtheorem{assumption}{Assumption}
\newtheorem{proposition}{Proposition}
\newtheorem{definition}{Definition}
\newtheorem{Property}{Property}
\newtheorem{theorem}{Theorem}
\newtheorem{lemma}{Lemma}

\theoremstyle{thmstyletwo}

\begin{document}
\title{Generalized Geometry Block Proximal Linearized Method for Multiblock Nonconvex and Nonsmooth Optimization}

\author*{\fnm{Weifeng} \sur{Yang}}\email{ywf841673182@gmail.com}

\abstract{


This paper considers a class of multiblock nonconvex and nonsmooth optimization problems arising in many applications.  
Existing methods construct proximal linearized operators or their variants within standard Euclidean geometry to solve this class of problems, forcing their block variable updates to rely on the standard inner product and its induced norm. 
Nevertheless, this construction fails to capture the geometric structure of the target problem, leading to low numerical efficiency. 
To overcome these drawbacks, we propose a generalized geometry proximal linearized operator for updating block variables, and develop the Generalized Geometry Block Proximal Linearized (GGBPL) method based on this operator. 
Compared with existing proximal linearized operators, the proposed operator allows the block surrogate functions to be constructed using arbitrary inner products and general admissible metrics, thereby enabling the GGBPL method to adapt its updates to the geometric structure of various  problems. We also introduce the inertial version of GGBPL, named the inertial GGBPL (iGGBPL) method.  
We further establish a new unified convergence framework under this generalized geometry, within which we prove that our methods guarantee convergence of the objective function values, establish global convergence of the generated sequence to a critical point, and derive the convergence rate of our methods. We also establish an $\mathcal{O}(\varepsilon^{-2})$ iteration complexity bound for obtaining an $\varepsilon$-stationary point. 
We apply our methods to two nonconvex and nonsmooth problems: sparse nonnegative matrix factorization with $\ell_0$-constraints and sparse nonnegative CP decomposition with $\ell_0$-constraints. 
Numerical results demonstrate the superior numerical performance of our proposed methods over several state-of-the-art methods. 

}


\maketitle

\section{Introduction}
In this paper, we consider a class of nonconvex and nonsmooth optimization problems as follows. 
\begin{align}
    \min\limits_{(\left\{{x_{i}} \right\}^{N}_{i=1})}J(\left\{{x_{i}} \right\}^{N}_{i=1})=H(\left\{{x_{i}} \right\}^{N}_{i=1})+\sum_{i=1}^{N}F_{i}(x_{i}),\label{e11}
\end{align}
where $d_{i} \in \mathbb{N}, ~H: \mathbb{D}_{H}  \rightarrow\mathbb{R},~\mathbb{D}_{H} =\prod_{i=1}^N \mathbb{R}^{d_{i}}, ~F_{i}:\mathbb{D}_{F_{i}} \rightarrow \overline{\mathbb{R}},~\mathbb{D}_{F_{i}} \subseteq \mathbb{R}^{d_{i}},~ \mathrm{dom} ~ J=\mathbb D_H\cap\prod_{i=1}^N\mathbb D_{F_i}$. $F_{i}$ is a proper, lower semicontinuous (possibly nonconvex and nonsmooth) function (e.g., $\ell_0$ norm), $H$ is a continuously differentiable (possibly nonconvex) function.  
Eq. (\ref{e11}) covers many application scenarios, e.g., sparse PCA \cite{li2023sparse,zhao2023proximal}, 
tensor decomposition \cite{zhang2022sparse,chen2022unsupervised}, matrix completion \cite{fan2020matrix,jia2022non}, etc.

Obviously, Eq. (\ref{e11}) represents a class of nonconvex and nonsmooth multiblock optimization problems. Commonly used methods for multiblock nonconvex optimization, including proximal alternating minimization (PAM) methods \cite{attouch2010proximal,dang2025two} and some smoothing methods \cite{li2024smoothing,chen2012smoothing}, struggle to provide closed-form block updates for Eq. (\ref{e11}) and incur high computational costs. 
To overcome these issues, the Proximal Alternating Linearized Minimization (PALM) method \cite{bolte2014proximal} introduces a block proximal linearized operator based on the standard Euclidean geometry as follows. 
\begin{equation}
\begin{aligned}
    && x^{k+1}_{j} &\in \operatorname*{\arg\min}\limits_{x} \Big[F_{j}(x)+\frac{1}{2\sigma^{k}_{j}} \|x-x^{k}_{j}\|^{2}\\
    && & +\langle  x-x^{k}_{j}, \nabla_{x_{j}} H(\left\{{x^{k+1}_{i}} \right\}^{j-1}_{i=1},x_{j}^{k},\left\{{x^{k}_{i}} \right\}^{N}_{i=j+1}) \rangle \Big]. 
\end{aligned}
\label{e12}
\end{equation}

The proximal linearized operator Eq. (\ref{e12}) is constructed under the standard inner product  (Euclidean inner product), and its proximal term is defined by the standard metric induced by this standard inner product (i.e., the Euclidean norm). 
Therefore, by utilizing the favorable mathematical properties of the Euclidean norm and the standard inner product, \cite{bolte2014proximal} proposed a convergence analysis framework based on the standard Euclidean geometry for the PALM-type methods. 
Following this convergence analysis framework, many researchers have also proposed numerous improved PALM-type methods. For example, since incorporating inertial terms (extrapolation terms) is an effective technique to accelerate convergence, subsequent researchers have proposed inertial PALM methods that incorporate extrapolation terms into Eq. (\ref{e12}). By imposing coupled constraints on the parameters (e.g., extrapolation parameters, step sizes, and Lipschitz parameters) and proving the descent of a Lyapunov auxiliary function instead of the original objective function, or by employing restart and backtracking steps, the convergence analyses of these inertial PALM methods can still be carried out within this standard Euclidean geometry convergence analysis framework, 
e.g.,  \cite{pock2016inertial} proposed an inertial PALM method with two extrapolation points, \cite{wang2024stochastic} proposed a stochastic version of the Gauss–Seidel type inertial PALM method for the case $N=2$, \cite{xu2017globally} introduced an inertial PALM method with backtracking steps, \cite{yang2023accelerated} proposed an inertial PALM method that uses restart steps to ensure the independence of the extrapolation parameter, among other methods not detailed here due to space limitations \cite{yang2024proximal,jia2025fast,wang2023generalized}. \cite{yashtini2022variable} also proposed a PALM-type method confined to a structurally restricted composite subclass of Eq. (\ref{e11}), and this method cannot handle the general problem class considered in this paper. 
Besides inertial PALM methods, some works also use Bregman distance regularization as the proximal term in Eq. (\ref{e12}) to extend the theoretical form of the proximal linearized operator. Similarly, in order to ensure convergence, these Bregman PALM methods construct the Bregman distance using a kernel function under the standard inner product, and require both the kernel function and its resulting Bregman distance to satisfy global strong convexity and locally Lipschitz gradient continuity with respect to the Euclidean norm, thereby allowing the convergence analyses of these Bregman PALM methods to still be carried out within the standard Euclidean geometry convergence analysis framework proposed by \cite{bolte2014proximal}, e.g., \cite{guo2023two} proposed a two-step inertial Bregman PALM method for the case $N=2$, \cite{le2020inertial} proposed an inertial Bregman PALM method, among other methods not detailed here due to space limitations \cite{ahookhosh2021multi,gao2023alternating,wang2024bregman,phan2023inertial}.



However, the above methods have to strictly confine the construction of their block updates to the standard Euclidean geometry.  
This restriction allows their convergence analyses to be carried out within this standard Euclidean geometry convergence analysis framework, but introduces geometric distortion and metric distortion when the geometries of the target problem are different from the Euclidean one. 
Moreover, there are some Bregman PALM methods (including those mentioned above) that use Bregman distance regularization as the proximal term in Eq. (\ref{e12}), but they still define the Bregman distance induced by the standard inner product and require it to be uniformly equivalent to the squared Euclidean norm, thereby also introducing geometric and metric distortion.  
Although different finite-dimensional topological spaces may be topologically equivalent, the geometric distortion and metric distortion between them can still be substantial \cite{bauschke2017descent,lu2018relatively}. 
Therefore, constructing all block updates within standard Euclidean geometry forces their proximal regularization and parameter selection to rely exclusively on the standard inner product and its induced norm, preventing these updates from capturing the local geometry, scaling, and coupling structures of the corresponding subproblems across different target problems,  thereby reducing the applicability of these methods and diminishing their numerical effectiveness.

To overcome these drawbacks, we propose a novel generalized geometry proximal linearized operator to update block variables and develop the Generalized Geometry Block Proximal Linearized (GGBPL) method based on this operator. 
The proposed operator constructs each block surrogate function using arbitrary inner products and general admissible metrics, enabling the proposed method to adapt its block updates to the local geometry, scaling, and coupling structures of the corresponding subproblems across different target problems, thereby improving the numerical efficiency and extending the applicability of the proposed GGBPL method. We further develop an inertial variant of GGBPL, termed iGGBPL. 
To provide theoretical guarantees in this generalized geometric setting, we further establish a new unified convergence framework. Within this framework, we prove that our methods guarantee convergence of the objective function values, establish global convergence of the generated sequence to a first-order critical point, and derive the convergence rate of our methods. We further establish a finite-iteration guarantee under the same generalized geometry, showing that our methods obtain an $\varepsilon$-stationary point within $\mathcal{O}(\varepsilon^{-2})$ iterations. 
Sparse nonnegative matrix factorization (SNMF) and sparse nonnegative CP decomposition (SNCP) are important tools for feature extraction \cite{li2025superpixel,chen2022unsupervised}, but their formulations with $\ell_0$ constraints are NP-hard, nonconvex and nonsmooth. Therefore, we apply our methods to solve these problems.

\subsection{Contributions}
The main contributions are summarized as follows.

(1) We propose a novel generalized geometry proximal linearized operator, thereby obtaining the Generalized Geometry Block Proximal Linearized (GGBPL) method based on this operator. 
Unlike existing proximal linearized operators strictly confined to standard Euclidean geometry, the proposed operator allows each block surrogate function to be constructed using arbitrary inner products and general admissible metrics. Consequently, the block updates can be adapted to the local geometry, scaling, and coupling structures of the corresponding subproblems across different target problems, thereby improving numerical efficiency and extending the applicability of the proposed methods. Additionally, we introduce an inertial version of GGBPL, named the inertial GGBPL (iGGBPL) method.

(2) We establish a unified convergence framework in a generalized geometric setting that accommodates arbitrary inner products and general admissible metrics.  
Within this framework, we prove that our methods guarantee convergence of the objective function values, establish global convergence of the generated sequence to a first-order critical point, and derive the corresponding convergence rate. We also establish an $\mathcal{O}(\varepsilon^{-2})$ iteration complexity bound for obtaining an $\varepsilon$-stationary point.

(3) We apply our methods to solve the SNMF with $\ell_0$-constraints and the SNCP  with $\ell_0$-constraints problems. Numerical results across all datasets and ranks demonstrate that GGBPL consistently outperforms PALM and several state-of-the-art inertial methods, while iGGBPL achieves the best overall performance by a clear margin among all compared methods.

\section{Symbol definitions and preliminaries}
\label{section: pre}
We first introduce some definitions and properties \cite{ciarlet2013linear,armstrong2013basic}. Additionally, Table \ref{notation} summarizes the notation used in this paper. 

\begin{table}[!ht] 
\centering 

\begin{tabular}{p{2.5cm}|p{7.3cm}} \hline%
Notation & Definition \\ \hline
$\left\{{x_{i}} \right\}^{n}_{i=1}$ & $\{x_{1}, x_{2},. . . . . . , x_{n}\}$ \\
$\left[ n \right]$  & $\left\{i\right\}_{i=1}^{n}$ \\

$x_{(n)}$  & $\left\{{x_{i}} \right\}^{n}_{i=1}$ \\ 
$x_{(n)}+y_{(n)}$ &	$\left\{{x_{i}}+{y_{i}}\right\}^{n}_{i=1}$ \\ 


$L_{\nabla_{x_{j}} H}$  & the Lipschitz constant of $\nabla_{x_{j}} H(\left\{{x_{i}} 
\right\}^{N}_{i=1})$\\

$x^{k}_{i}$  & the $i$-th block of $\left\{{x_{i}} \right\}^{n}_{i=1}$ within the $k$-th outer loop\\

$h_{j}(x_{j})$ &$H(\{x_{i}^{k+1}\}_{i=1}^{j-1}, x_{j}, \{x_{i}^{k}\}_{i=j+1}^{N}) $ \\

$\mathcal{C}^{1}_{L}(X)$ & the set of functions satisfying the block-wise local gradient Lipschitz continuity \\

$\mathrm{dom} ~ J$ & the domain of function $J$ \\

$\langle \cdot , \cdot  \rangle$ & standard inner product (dot product) \\

$\langle\cdot,\cdot\rangle_{i,k}$ & an arbitrary inner product defined on the $i$-th block at iteration $k$ \\

$\|\cdot\|_{i,k}$ & the norm induced by $\langle\cdot,\cdot\rangle_{i,k}$ \\


$\mathcal{D}_i$ & the family of admissible metrics for the $i$-th block, it satisfies Assumption \ref{assump2} \\

$d_i^k(\cdot,\cdot)$ & the admissible metric $d_i^k\in\mathcal{D}_i$ for the $i$-th block at iteration $k$  \\  

$L_{i,k}$ & the local majorization constant of $h_i$ under $d_i^k$ \\ \hline

\end{tabular}

\caption{Summary of frequently used notations}
\label{notation}

\end{table}

\begin{definition} 
Let $X$ be a non-empty set. A function $d:X\times X\rightarrow\mathbb{R}$ is called a metric on $X$ if, for all $x, y, z \in X$, the following conditions hold. 

(i) $d(x, y) \geq 0$, with equality if and only if $x = y$.

(ii) $d(x, y) = d(y, x)$ for all $x, y \in X$.

(iii) $d(x, y) + d(y, z) \geq d(x, z)$.

\noindent Then $(X,d)$ is a metric space, we abbreviate it as $X$. 
\label{d0}
\end{definition}

\begin{proposition}
Let $(X,d)$ be a metric space. Then the metric function $d:X\times X\rightarrow \mathbb{R}$ is continuous with respect to the topology induced by $d$. That is, if $x_k\rightarrow x$ and $y_k\rightarrow y$ in $(X,d)$, then $\lim_{k\rightarrow \infty}d(x_k,y_k)=d(x,y).$
\label{p_metric}
\end{proposition}

\begin{definition}
$B(x,\epsilon)$ is an open ball which is defined as $B(x,\epsilon):=\left\{y: d(x,y)<\epsilon\right\}$. 
\label{openball}
\end{definition}

\begin{proposition}
A bounded closed set in finite-dimensional space is a compact set. 
\label{dcompact}
\label{p4}
\end{proposition}

\begin{proposition}
Let $f:\mathbb{R}^{n}\rightarrow\mathbb{R}$ and $x_0\in\mathbb{R}^{n}$. Then
\begin{center}
$\lim_{x\rightarrow x_0}f(x)=f(x_0)\Leftrightarrow({\forall}\{x_k\}_{k\in\mathbb{N}}\subseteq\mathbb{R}^{n},~\lim_{k\rightarrow\infty}x_k=x_0\Rightarrow\lim_{k\rightarrow\infty}f(x_k)=f(x_0))$.
\end{center}
\label{tf4}
\end{proposition}

\begin{definition}
For $S \subseteq \mathbb{R}^{n}$, we define the distance between the set $S$ and the point $x$ as $\operatorname{dist}(x, S): =\operatorname*{inf}\,\left\{\|x-y\|: \,y\in S\right\}. $
\end{definition}

\subsection{Notation and preliminaries for nonconvex analysis}
Next, we introduce some preliminaries for nonconvex analysis \cite{bolte2014proximal,ciarlet2013linear}.  
\begin{definition}
Proper function: a function $g: \mathbb{R}^{n} \rightarrow (-\infty,+\infty]$ is said to be proper if $\mathrm{dom}$ $g \neq \emptyset$, where $\mathrm{dom}~g=\{x\in\mathbb{R}^{n}:g(x)<\infty\}$. 
\label{dproper}
\end{definition}
\begin{definition}
A function $f:\mathbb{R}^{n}\rightarrow(-\infty,+\infty]$ is lower semicontinuous if, for every sequence $\{x_k\}_{k\in\mathbb{N}}\subseteq\mathbb{R}^{n}$ satisfying $x_k\rightarrow x$, we have
\begin{center}
$f(x)\leq\liminf_{k\rightarrow\infty}f(x_k)$.
\end{center}
\label{lower}
\end{definition}

\begin{definition}
Coercive function: a function $f:\mathbb{R}^{n}\rightarrow(-\infty,+\infty]$ is coercive if $f(x)\rightarrow+\infty$ as $\|x\|\rightarrow\infty$. Equivalently, for every $a\in\mathbb{R}$, the sublevel set $\{x\in\mathbb{R}^{n}:f(x)\leq a\}$ is bounded. 
\label{dlower}
\end{definition}
\begin{definition}
Let $f$ be a proper lower semicontinuous function, The Fréchet subdifferential of $f$ at $x$, written ${\hat{\partial}} f(x)$, is the set of all vectors u which satisfy  
\begin{center}
$\liminf_{y\neq x,y\to x}\frac{f(y)-f(x)-\langle u,\ y-x\rangle}{\|y-x\|}\ge0, $
\end{center}
when $ x \notin \mathrm{dom}~ f$, then set ${\hat{\partial}} f(x)=\emptyset$.
\label{d2}
\label{dsubdiff}
\end{definition}
\begin{definition}
The limiting subdifferential  $\partial f(x): =\{u\in\mathbb{R}^{n}: \exists x^{k}\to x,f(x^{k})\to f(x),u^{k}\to u,u^{k}\in\widehat{\partial}f(x^{k})\}.$
\end{definition}
\begin{proposition}
Let $f$ be a proper lower semicontinuous function. If $f$ has a local minimum at $x^{*}$, then $0 \in \partial f(x^{*})$. 
\label{p1}
\end{proposition}
\begin{proposition} 
Let $f$ be a proper lower semicontinuous function, and $g$ be a continuously differentiable function. Then $\forall x \in \mathrm{dom} ~f$, $\partial(f+g)(x)$ = $\partial f (x) +\nabla g(x)$. 
\label{p2}
\end{proposition}

\begin{definition} 
Set $f:X\rightarrow \mathbb{R}, ~ X \subseteq \mathbb{R}^{n}$. We say that $f\in \mathcal{C}_{L}^{1}(X)$ if, for any $x_i$-section $X_i$ of $X$ and the corresponding $x_i$-section $f_i$ of $f$, $\exists L_{\nabla_{x_{i}} f}>0$, $\forall y_i,z_i \in X_i$, s.t. 
\begin{equation} 
\|\nabla_{x_{i}} f_i(z_i)-\nabla_{x_{i}} f_i(y_i)\| \leq L_{\nabla_{x_{i}} f}\|z_i-y_i\|,  \notag 
\end{equation} 
and, for any bounded set $B\subseteq X$, $\exists L_{\nabla f}>0$, $\forall y,z\in B$, s.t. 
\begin{equation} 
\|\nabla f(z)-\nabla f(y)\| \leq L_{\nabla f}\|z-y\|. \notag 
\end{equation} 
Here, $L_{\nabla_{x_{i}} f}$ is the block-wise Lipschitz constant of $\nabla_{x_{i}} f$, and $L_{\nabla f}$ is the Lipschitz constant of $\nabla f$ on $B$. $\mathcal{C}_{L}^{1}(X)$ is the set of functions satisfying the block-wise gradient Lipschitz continuity and the gradient Lipschitz continuity on bounded sets. 
\label{dlipchitz} 
\label{d3} 
\end{definition}

\begin{proposition} 
Set $f:X\rightarrow \mathbb{R}, ~ X\subseteq\mathbb{R}^{n}$. If $f\in\mathcal{C}_{L}^{1}(X)$, then, for any $x_i$-section $X_i$ of $X$ and the corresponding $x_i$-section $f_i$ of $f$, $\forall x_i,y_i\in X_i$ satisfying $\{(1-t)x_i+ty_i:t\in[0,1]\}\subseteq X_i$, s.t. 
\begin{equation}
f_i(y_i) \leq f_i(x_i)+\langle \nabla f_i(x_i),y_i-x_i \rangle +\frac{L_{\nabla_{x_{i}} f}}{2}\|y_i-x_i\|^{2}. \notag 
\end{equation}
\label{p3} 
\end{proposition}

\begin{proposition}
All norms on a finite-dimensional space are equivalent. 
That is, for any two norms $\|\cdot\|_{a}$ and $\|\cdot\|_{b}$ on $\mathbb{R}^{n}$, there exist constants $c_{1},c_{2}>0$ such that
$c_{1}\|x\|_{a}\leq \|x\|_{b}\leq c_{2}\|x\|_{a}, \quad \forall x\in\mathbb{R}^{n}$.
\label{p_norm_equiv}
\end{proposition}

\begin{proposition}
Every finite-dimensional normed space is complete. 
Therefore, every Cauchy sequence in a finite-dimensional normed space is convergent.
\label{p_complete}
\end{proposition}

\begin{proposition}
Set $f:X\rightarrow \mathbb{R},~X\subseteq\mathbb{R}^{n}$. 
If $f\in\mathcal{C}_{L}^{1}(X)$, then, for any $x_i$-section $X_i$ of $X$ and the corresponding $x_i$-section $f_i$ of $f$, $\forall x_i,y_i\in X_i$ satisfying $\{(1-t)x_i+ty_i:t\in[0,1]\}\subseteq X_i$, any selected inner product $\langle\cdot,\cdot\rangle_{i,k}$, and any admissible metric $d_i^k$ on the $i$-th block space, there exists a constant $L_{i,k}>0$ such that
\begin{align}
f_i(y_i)
&\leq
f_i(x_i)
+\langle R_{i,k}^{-1}\nabla_i f_i(x_i),y_i-x_i\rangle_{i,k}
+\frac{L_{i,k}}{2}(d_i^k(y_i,x_i))^{2}.
\notag
\end{align}
where $\|z\|_{i,k}=\sqrt{\langle z,z\rangle_{i,k}}$, and $R_{i,k}^{-1}\nabla_i f_i(x_i)$ denotes the gradient of $f_i$ with respect to $\langle\cdot,\cdot\rangle_{i,k}$.
\label{p3}
\end{proposition}

\begin{proposition}
Let $f:\mathbb{R}^{n}\rightarrow(-\infty,+\infty]$ be a proper lower semicontinuous convex function. 
Then $\partial f$ is locally bounded on $\operatorname{int}(\operatorname{dom} f)$. 
That is, for any compact set $K\subset \operatorname{int}(\operatorname{dom} f)$, there exists $C_K>0$ such that $\|u\|\leq C_K,\quad \forall x\in K,\quad \forall u\in \partial f(x).$
\label{p_sub_bound}
\end{proposition}

\begin{definition}
Desingularization function: we denote the desingularization function $\phi: [0,\eta)\rightarrow \mathbb{R}_{+}$ which satisfies the following conditions. 

(i) $\phi(0)=0$.

(ii) $\phi$ is continuous on $[0,\eta)$, continuously differentiable on
$(0,\eta)$, and satisfies $\phi'(s)>0$ for every $s\in(0,\eta)$. 

(iii)  $\phi$ is concave on $[0,\eta)$. 
\label{des}
\end{definition}

The Kurdyka-Łojasiewicz (KŁ) property \cite{attouch2013convergence,xu2013block}, described below, serves as a key tool for global convergence analysis and establishing convergence rate.  
\begin{definition}
Let $f$ be a proper lower semicontinuous function. The function $f$ is said to have the KŁ property at $\bar{u}\in\operatorname{dom}(\partial f)$ if there exist $\eta\in\left(0,+\infty\right]$, a neighborhood $U$ of $\bar{u}$, and a desingularization function $\phi$ such that, for every $u\in U$ satisfying $f(\bar{u})<f(u)<f(\bar{u})+\eta$, we have
\begin{center}
$\phi^{\prime}(f(u)-f(\bar{u}))\operatorname{dist}(0,\partial f(u))\geq1$.
\end{center}
If $f$ has the KŁ property at each point of $\operatorname{dom}(\partial f)$, then $f$ is a KŁ function. 
\label{def_KL}
\end{definition}
\begin{proposition}
(Uniformized KŁ property) Let $\Omega$ be a compact set and $f$:  $\mathbb{R}^{n} \rightarrow \mathbb{R} \cup \left\{\infty\right\}$ be a proper and lower semicontinuous function. Assume that $f$ is constant on $\Omega$ and is a KŁ function. Then, there exist $\epsilon>0$, $\eta>0$, and a desingularization function $\phi$ such that, for every $\bar{x}\in\Omega$ and every $x$ satisfying
$\operatorname{dist}(x,\Omega)<\epsilon$ and
$f(\bar{x})<f(x)<f(\bar{x})+\eta$, we have
\begin{center}
$\phi^{\prime}(f(x)-f(\bar{x}))\operatorname{dist}(0,\partial f(x))\geq1$.
\end{center}
\label{def_uniKL}
\end{proposition}


\section{The proposed methods}
\label{section: A}
In this section, we present the technical details of our proposed generalized geometry proximal linearized operator and the GGBPL and iGGBPL methods.

First, we make the following assumptions about Eq. (\ref{e11}). 
\begin{assumption}
(i) $\mathbb{D}_{H} =\prod_{i=1}^N \mathbb{R}^{d_{i}}$, $H:\mathbb{D}_{H}\rightarrow\mathbb{R}$, $H\in\mathcal{C}_{L}^{1}(\mathbb{D}_{H})$. 

(ii) $F_{i}: \mathbb{D}_{F_{i}} \rightarrow \overline{\mathbb{R}}$ is a proper, lower semicontinuous function with a lower bound. $F_i(x)=+\infty$ for $x\notin\mathbb{D}_{F_i}$. 

(iii) $J$ is a Kurdyka-Łojasiewicz (KŁ) function, and $J$ is also a coercive proper function with a lower bound. 
\label{assump1}
\end{assumption}

Assumption \ref{assump1} is the standard assumption for Eq. (\ref{e11}). We also assume the following conditions for the metrics of our proposed methods.

\begin{assumption}
(i) There exist a norm $\|\cdot\|_d$ and a positive constant $C_d$ s.t. $C_d\|x-y\|_d\leq d_i^k(x,y)$ ($\forall x,y\in\mathbb{R}^{d_i}$, $\forall d_i^k(\cdot,\cdot)\in\mathcal{D}_i$).

(ii) $d_i^k(\cdot,y)$ is convex ($\forall y\in\mathbb{R}^{d_i}$).

(iii) $\forall i\in[N]$, define the upper envelope $\omega_i(x,y):=\sup_{d\in\mathcal{D}_i}d(x,y)$. Then, $\omega_i$ is bounded on every compact  subset of $\mathbb{R}^{d_i}\times\mathbb{R}^{d_i}$. 
\label{assump2}
\end{assumption}

Then, we replace the standard inner product and the squared Euclidean norm proximal term in Eq. (\ref{e12}) with an arbitrary inner product $\langle\cdot,\cdot\rangle_{i,k}$ and a squared admissible metric $(d_i^k(\cdot,\cdot))^2$. This yields our generalized geometry proximal linearized operator, which can be expressed as follows. 
\begin{equation}
\begin{aligned}
&&x^{k+1}_{i} &\in Gprox_{\sigma_i^kF_i}^{d_i^k}
\left(
x_i^k;
R_{i,k}^{-1}\nabla_{x_i}h_i(x_i^k)
\right)  \\
&& &:= \operatorname*{\arg\min} \limits_{x \in \mathbb{D}_{F_{i}}} \Big[F_{i}(x)+\frac{1}{2\sigma^{k}_{i}}(d_i^k(x,x^{k}_{i}))^{2}  \\
&& &+\langle R_{i,k}^{-1}\nabla_{x_{i}} H(\{x^{k+1}_{j}\}_{j=1}^{i-1},x^k_{i},\{x^{k}_{j}\}_{j=i+1}^{N}), x-x^{k}_{i} \rangle_{i,k}   \Big], 
\end{aligned}
\label{gprox} 
\end{equation}
where $\langle u,v\rangle_{i,k}=\langle R_{i,k}u,v\rangle$, $R_{i,k}$ is the Riesz map. Let $h_{i}(c)=H(\{x^{k+1}_{j}\}_{j=1}^{i-1},c,\{x^{k}_{j}\}_{j=i+1}^{N})$, next, we present our method: GGBPL (Algorithm \ref{GGBPL}).

\begin{footnotesize}
 \begin{algorithm}[!h]
\caption{GGBPL: Generalized Geometry Block Proximal Linearized method}\label{GGBPL}
    \SetAlgoLined
     \LinesNumbered
        \KwIn{$\{{x^{1}_{i}} \}^{N}_{i=1}=\{{x^{0}_{i}} \}^{N}_{i=1} \in \mathrm{dom} \ J$}
        
        
        \For {$k = 1,2,3,\ldots$, $k_{max}$ }{

        \For {$i = 1$ to N}{

        $s_i^k\in(L_{i,k},\infty),~\gamma^{k}_{i}\in (1,\infty), ~\sigma_{i}^{k} =\frac{1}{\gamma^{k}_{i} s_i^k}$,

         
        $x^{k+1}_{i} \in Gprox_{\sigma^{k}_{i} F_{i}}^{d_i^k} \left(x^{k}_{i};R_{i,k}^{-1}\nabla_{x_{i}} h_{i}(x^{k}_{i})\right).$
         
        }

        }
        \textbf{Return} $\{{x^{k+1}_{i}}\}^{N}_{i=1}$.
        
\end{algorithm}
\end{footnotesize}


As outlined in the Introduction, incorporating the inertial term is a popular and effective first-order accelerated convergence method. To accelerate convergence of the GGBPL algorithm, we also propose its inertial version combined with the inertial term: iGGBPL (Algorithm \ref{iGGBPL}).

\begin{footnotesize}
 \begin{algorithm}[!h]
\caption{iGGBPL: inertial Generalized Geometry Block Proximal Linearized method}\label{iGGBPL}
    \SetAlgoLined
     \LinesNumbered
        \KwIn{$\{{x^{1}_{i}} \}^{N}_{i=1}=\{{x^{0}_{i}} \}^{N}_{i=1} \in \mathrm{dom} \ J$,  $\rho_1>0$, $\beta^{1}=0$, $t_1=1$}
        
        
        \For {$k = 1,2,3,\cdots, k_{max}$ }{

        $\{{y^{k}_{i}} \}^{N}_{i=1}$=$\{{x^{k}_{i}} \}^{N}_{i=1}$+$\beta^{k}\times (\{{x^{k}_{i}} \}^{N}_{i=1}-\{{x^{k-1}_{i}} \}^{N}_{i=1})$. \\

        \For {$i = 1$ to N}{

        $s_i^k\in(L_{i,k},\infty),~\gamma^{k}_{i}\in (1,\infty), ~\sigma_{i}^{k} = \frac{1}{\gamma^{k}_{i} s_i^k}$,

        $x^{k+1}_{i} \in Gprox_{\sigma^{k}_{i} F_{i}}^{d_i^k} \left(y^{k}_{i};R_{i,k}^{-1}\nabla_{x_{i}} h_{i}(y^{k}_{i})\right).$

        }

        \If{$J(x^{k+1}_{(N)}) > J(x^{k}_{(N)})-\rho_{1}\sum_{i=1}^N (d_i^k(x^{k+1}_{i},x^{k}_{i}))^2$} 
        {$\beta^{k}=0$, re-update $x^{k+1}_{(N)}$ through steps 2 to 6.}

        $t_{k+1}=\frac{1+\sqrt{1+4t_{k}^{2}}}{2}$, $\beta^{k+1}=\frac{t_{k}}{t_{k+1}}$.
        }

        \textbf{Return} $\{{x^{k+1}_{i}}\}^{N}_{i=1}$.
        
\end{algorithm}
\end{footnotesize}



\section{Convergence analysis}
\label{sec4}
In this section, we demonstrate that our proposed methods guarantee convergence of the objective function values, establish the global convergence of the generated sequence to a first-order critical point, as well as derive the convergence rate of our proposed methods. 
Since the proposed methods construct block updates under general admissible metrics, the convergence analysis has to be carried out in a generalized geometric setting rather than within the standard Euclidean geometry. 
Consequently, the standard Euclidean geometry convergence analysis framework is not applicable to this broader geometry.  
Therefore, we develop a new convergence analysis framework under this generalized geometric setting.

\begin{Property}[Convergence analysis framework in a generalized geometric space]  
The main steps for proving the convergence properties, including the global convergence of the generated sequence, in this generalized geometric setting are as follows.

(\romannumeral1) There exists a positive constant $\rho$ such that 
$\rho\sum_{i=1}^N (d_i^k(x^{k+1}_{i},x^{k}_{i}))^2\leq J(x^k_{(N)})-J(x^{k+1}_{(N)})$ .

(\romannumeral2) There exists a positive constant $\rho_d $ such that $\exists\eta_i^{k+1}\in\partial_1d_i^k(x_i^{k+1},y^k_i)$ satisfying
$-\nabla_i h_i(y^k_i)-\frac{d_i^k(x_i^{k+1},y^k_i)}{\sigma_i^k}\eta_i^{k+1}\in\partial F_i(x_i^{k+1})$
and $\|\eta_i^{k+1}\|\leq \rho_d $.

(\romannumeral3) There exists a positive constant $\rho_{b}$ such that 
$\rho_{b}\sum_{i=1}^N (d_i^k(x^{k+1}_{i},y^{k}_{i})+d_i^k(x_i^{k+1},x^k_i))\geq \|p_{x^{k+1}}\|$, where $p_{x^{k+1}} \in \partial J(x^{k+1}_{(N)})$.


(\romannumeral4) By combining the favorable metric and topological properties of finite-dimensional spaces, the generated sequence $\left\{x^{k}_{(N)}\right\}_{k \in \mathbb{N}}$ converges to a point $x^* \in  \mathrm{dom} ~ J$ within this metric space. 
\label{based}
\end{Property}

When $d_i^k(x_i,y_i)$ reduces to the standard Euclidean norm and $\langle\cdot,\cdot\rangle_{i,k}$  reduces to the standard inner product, the proposed convergence analysis framework (Property \ref{based}) reduces to the standard Euclidean geometry convergence analysis framework used by PALM-type methods. Consequently, the proposed framework subsumes the PALM-type Euclidean convergence analysis as a special case and extends it to a generalized geometric setting.


Notably, Algorithm \ref{GGBPL} is recovered from Algorithm \ref{iGGBPL} by setting $\beta^k\equiv0$. Therefore, it suffices to analyze the convergence of Algorithm \ref{iGGBPL}, since the convergence results for Algorithm \ref{GGBPL} follow directly as a special case. Throughout the convergence analysis, the step sizes are selected such that $\inf_{i\in[N],~k\geq1}(\gamma_i^k-1)s_i^k>0$ and $\inf_{i\in[N],~k\geq0}\sigma_i^k>0$.

\subsection{Sufficient decrease of the objective function}
We first prove Property \ref{based}(\romannumeral1). Since the proposed method is constructed under arbitrary inner products and admissible metrics, the convergence analysis has to be carried out in a generalized metric geometric setting rather than within the standard Euclidean geometry. 
Therefore, we first establish the sufficient decrease property of the objective function $J$ under this generalized metric geometric setting, and then derive the global convergence of the generated sequence. 

\begin{theorem} 
Let $\left\{x^{k}_{(N)}\right\}_{k \in \mathbb{N}}$ be the sequence generated by Algorithm \ref{iGGBPL}. \\ 
(\romannumeral1) 
$J(x^{k+1}_{(N)})\leq J(x^{k}_{(N)}) -\rho \sum_{i=1}^N(d_i^k(x^{k+1}_{i},x^{k}_{i}))^{2}$,    
where $\rho$ is a positive constant.  

(\romannumeral2) 
$\lim_{k \to \infty} d_i^k(x^{k+1}_{i},x^{k}_{i})=0$.

\label{t1}
\end{theorem}

\begin{proof}
(\romannumeral1) 
We first define 
\begin{center}
    $h_{j}(x_{j})=H(\{x_{i}^{k+1}\}_{i=1}^{j-1}, x_{j}, \{x_{i}^{k}\}_{i=j+1}^{N})$. 
\end{center}

If $J(x^{k+1}_{(N)}) \leq J(x^{k}_{(N)})-\rho_{1}\sum_{i=1}^N (d_i^k(x^{k+1}_{i},x^{k}_{i}))^2$, then Theorem \ref{t1}(i) is true. If not, from Proposition \ref{p3}, we have
\begin{align}
&&H(x_{1}^{k+1},\{x_{j}^{k}\}_{j=2}^{N})
&\leq H(x^{k}_{(N)})
+\left\langle R_{1,k}^{-1}\nabla_{x_{1}}H(x^{k}_{(N)}),x_{1}^{k+1}-x_{1}^{k}\right\rangle_{1,k}\notag \\
&& &+\frac{L_{1,k}}{2}\left(d_{1}^{k}(x_{1}^{k+1},x_{1}^{k})\right)^{2}.
\label{e331}
\end{align}
From Eq. (\ref{gprox}), we obtain 
\begin{align}
&&F_{1}(x_{1}^{k})
&\geq F_{1}(x_{1}^{k+1})
+\left\langle R_{1,k}^{-1}\nabla_{x_{1}}H(x^{k}_{(N)}),x_{1}^{k+1}-x_{1}^{k}\right\rangle_{1,k} +\frac{1}{2\sigma_{1}^{k}}(d_{1}^{k}(x_{1}^{k+1},x_{1}^{k}))^{2}.
\label{e332}
\end{align}
Then sum of Eq. (\ref{e331}) and Eq. (\ref{e332}), we have 
\begin{align}
&&H(x^{k}_{(N)})+F_{1}(x_{1}^{k})&\geq F_{1}(x_{1}^{k+1})+H(x_{1}^{k+1},\{x_{j}^{k}\}_{j=2}^{N})+(\frac{1}{2\sigma_{1}^{k}}-\frac{L_{1,k}}{2})(d_{1}^{k}(x_{1}^{k+1},x_{1}^{k}))^{2}.
\end{align}

Since $\sigma_j^k=\frac{1}{\gamma_j^k s_j^k}$, $\gamma_j^k>1$, and $s_j^k>L_{j,k}$, we have
$\frac{1}{2\sigma_j^k}-\frac{L_{j,k}}{2}=\frac{\gamma_j^ks_j^k-L_{j,k}}{2}>\frac{(\gamma_j^k-1)s_j^k}{2}$. Assume that the result holds when $i=n$, i.e.,
\begin{align}
&&&J(\{x_{j}^{k+1}\}_{j=1}^{n},\{x_{j}^{k}\}_{j=n+1}^{N})\leq J(x^{k}_{(N)})
-\sum_{j=1}^{n}\frac{(\gamma_j^k-1)s_j^k}{2}
(d_j^k(x_j^{k+1},x_j^k))^2.
\label{e333}
\end{align}

Since Eq. (\ref{gprox}), we also have  
\begin{align}
&&F_{n+1}(x_{n+1}^{k})&\geq F_{n+1}(x_{n+1}^{k+1})+\langle R_{n+1,k}^{-1}\nabla_{x_{n+1}} h_{n+1}(x_{n+1}^{k}),x_{n+1}^{k+1}-x_{n+1}^{k}\rangle_{n+1,k}\notag \\
&& &+\frac{1}{2\sigma_{n+1}^{k}}
(d_{n+1}^{k}(x_{n+1}^{k+1},x_{n+1}^{k}))^2.
\label{e334}
\end{align}

From Proposition \ref{p3}, we infer
\begin{align}
&&H(\{x_{j}^{k+1}\}_{j=1}^{n+1},\{x_{j}^{k}\}_{j=n+2}^{N})&\leq H(\{x_{j}^{k+1}\}_{j=1}^{n},\{x_{j}^{k}\}_{j=n+1}^{N})\notag \\
&&&+\langle R_{n+1,k}^{-1}\nabla_{x_{n+1}} h_{n+1}(x_{n+1}^{k}),x_{n+1}^{k+1}-x_{n+1}^{k}\rangle_{n+1,k} \notag \\
&&  &+\frac{L_{n+1,k}}{2}
(d_{n+1}^{k}(x_{n+1}^{k+1},x_{n+1}^{k}))^2.
\label{e335}
\end{align}

Then,  sum of Eq. (\ref{e333}), Eq. (\ref{e334}) and Eq. (\ref{e335}), when $n+1=N$, we have  
\begin{align}
&&&J(x^{k+1}_{(N)})
\leq
J(x^{k}_{(N)})
-\sum_{j=1}^{N}\frac{(\gamma_j^k-1)s_j^k}{2}
(d_j^k(x_j^{k+1},x_j^k))^2.
\label{c1end2}
\end{align}

Let
\begin{align}
\rho
:=
\min\left\{
\rho_1,\,
\frac{1}{2}
\inf_{i\in[N],\,k\geq1}
(\gamma_i^k-1)s_i^k
\right\}
>0.
\notag
\end{align}
If the inertial update is accepted, the acceptance condition gives
\begin{align}
J(x^{k+1}_{(N)})
\leq
J(x^k_{(N)})
-\rho\sum_{i=1}^N
(d_i^k(x_i^{k+1},x_i^k))^2.
\notag
\end{align}
If the inertial update is rejected, Eq. (\ref{c1end2}) gives the same inequality.


(\romannumeral2) From Theorem \ref{t1}(\romannumeral1), we have
\begin{align}
\rho\sum_{j=1}^{N}
(d_j^k(x_j^{k+1},x_j^k))^2
\leq
J(x^k_{(N)})-J(x^{k+1}_{(N)}).
\notag
\end{align}
Summing both sides from $k=1$ to $\infty$, and since $J$ has a lower bound, we have
\begin{align}
&&\rho\sum_{k=1}^{\infty}\sum_{j=1}^{N}(d_j^k(x_j^{k+1},x_j^k))^2 &\leq
\sum_{k=1}^{\infty}(J(x^{k}_{(N)})-J(x^{k+1}_{(N)})) \notag\\
&&&\leq J(x^{1}_{(N)})-\inf J \notag\\
&&&<\infty.\notag
\end{align}
Therefore, we have
\begin{align}
&&&\lim_{k\to\infty}\sum_{j=1}^{N}\rho(d_j^k(x_j^{k+1},x_j^k))^2=0,\notag
\end{align}
Since each term 
$\rho(d_j^k(x_j^{k+1},x_j^k))^2$ is nonnegative, we have
\begin{align}
&&&\lim_{k\to\infty}
\rho(d_j^k(x_j^{k+1},x_j^k))^2=0,~\forall j\in[N], 
\notag
\end{align}
which implies 
\begin{align}
&&&\lim_{k\to\infty}d_j^k(x_j^{k+1},x_j^k)=0,~ \forall j\in[N].
\notag
\end{align}
By Assumption \ref{assump2}, we have 
\begin{align}
&&&\lim_{k\to\infty}C_d\|x_j^{k+1}-x_j^k\|_d\leq \lim_{k\to\infty}d_j^k(x_j^{k+1},x_j^k), ~ \forall j\in[N].\notag
\end{align}
Thus, we have 
\begin{align}
&&&\lim_{k\to\infty}\|x^{k+1}_{(N)}-x^k_{(N)}\|_d=0.
\notag
\end{align}

\end{proof}

\subsection{Subgradient bound and cluster point characterization}
We next prove Property \ref{based}(\romannumeral2) and (\romannumeral3).

\begin{lemma}
\label{metric_bound}
Suppose that Assumption \ref{assump2}(ii) and (iii) hold and the sequences $\{x_i^{k+1}\}_{k\geq0}$ and $\{y_i^k\}_{k\geq0}$ generated by Algorithm \ref{iGGBPL} are bounded. Let $\underline{\sigma}:=\inf_{i\in[N],~k\geq0}\sigma_i^k>0$. Then, there exists a positive constant $\rho_d$ s.t., $\forall i\in[N]$, $k\geq0$, and $\eta_i^{k+1}\in\partial_1d_i^k(x_i^{k+1},y_i^k)$, we have 
\begin{align}
\|
\frac{d_i^k(x_i^{k+1},y_i^k)}
{\sigma_i^k}
\eta_i^{k+1}
\|
\leq
\rho_d
d_i^k(x_i^{k+1},y_i^k).
\notag
\end{align}
\end{lemma}

\begin{proof}
Fix $i\in[N]$. Since $\{x_i^{k+1}\}_{k\geq0}$ and $\{y_i^k\}_{k\geq0}$ are bounded, there exists $R_i>0$ s.t.
\begin{align}
\|x_i^{k+1}\|_d
\leq
R_i,
\qquad
\|y_i^k\|_d
\leq
R_i,
\quad
\forall k\geq0.
\notag
\end{align}
Define the compact sets
\begin{align}
K_i
&:=
\{u\in\mathbb{R}^{d_i}:\|u\|_d\leq R_i+1\},
\notag\\
\widehat{K}_i
&:=
\{y\in\mathbb{R}^{d_i}:\|y\|_d\leq R_i\}.
\notag
\end{align}
By Assumption \ref{assump2}(iii),
\begin{align}
M_i
:=
\sup_{\substack{u\in K_i~y\in\widehat{K}_i}}
\omega_i(u,y)
=
\sup_{\substack{d\in\mathcal{D}_i\\u\in K_i,~y\in\widehat{K}_i}}
d(u,y)
<
+\infty.
\notag
\end{align}

Let $\eta_i^{k+1}\in\partial_1d_i^k(x_i^{k+1},y_i^k)$ and take any $v\in\mathbb{R}^{d_i}$ satisfying $\|v\|_d\leq1$. By Assumption \ref{assump2}(ii),
\begin{align}
d_i^k(x_i^{k+1}+v,y_i^k)
&\geq
d_i^k(x_i^{k+1},y_i^k)
+
\langle\eta_i^{k+1},v\rangle,
\notag\\
d_i^k(x_i^{k+1}-v,y_i^k)
&\geq
d_i^k(x_i^{k+1},y_i^k)
-
\langle\eta_i^{k+1},v\rangle.
\notag
\end{align}
Since $x_i^{k+1}+v,x_i^{k+1}-v\in K_i$ and $y_i^k\in\widehat{K}_i$, we have
\begin{align}
d_i^k(x_i^{k+1}+v,y_i^k)
\leq
M_i,
\qquad
d_i^k(x_i^{k+1}-v,y_i^k)
\leq
M_i.
\notag
\end{align}
Moreover, $d_i^k(x_i^{k+1},y_i^k)\geq0$. Combining the above inequalities gives
\begin{align}
-M_i
\leq
\langle\eta_i^{k+1},v\rangle
\leq
M_i,
\quad
\forall\|v\|_d\leq1.
\notag
\end{align}
Therefore,
\begin{align}
\|\eta_i^{k+1}\|_{d,*}
=
\sup_{\|v\|_d\leq1}
|\langle\eta_i^{k+1},v\rangle|
\leq
M_i,
\quad
\forall k\geq0.
\notag
\end{align}

Since $\mathbb{R}^{d_i}$ is finite-dimensional, there exists $\kappa_i>0$ s.t.
\begin{align}
\|z\|
\leq
\kappa_i\|z\|_{d,*},
\quad
\forall z\in\mathbb{R}^{d_i}.
\notag
\end{align}
It follows that
\begin{align}
\|
\frac{d_i^k(x_i^{k+1},y_i^k)}
{\sigma_i^k}
\eta_i^{k+1}
\|
&\leq
\frac{\kappa_i}{\sigma_i^k}
d_i^k(x_i^{k+1},y_i^k)
\|\eta_i^{k+1}\|_{d,*}
\notag\\
&\leq
\frac{\kappa_iM_i}{\underline{\sigma}}
d_i^k(x_i^{k+1},y_i^k).
\notag
\end{align}
Since $N$ is finite, define
\begin{align}
\rho_d
:=
\max_{i\in[N]}
\frac{\kappa_iM_i}{\underline{\sigma}}.
\notag
\end{align}
Then,
\begin{align}
\|
\frac{d_i^k(x_i^{k+1},y_i^k)}
{\sigma_i^k}
\eta_i^{k+1}
\|
\leq
\rho_d
d_i^k(x_i^{k+1},y_i^k),
\quad
\forall i\in[N],\quad
\forall k\geq0.
\notag
\end{align}
\end{proof}

\begin{theorem}
In Algorithm \ref{iGGBPL}, if we define
\begin{align}
 && &p_{x^{k+1}}=\{\nabla_{x_{j}}H (x^{k+1}_{(N)})-\nabla_{x_{j}}h_j (y^{k}_{j})-\frac{d_j^k(x^{k+1}_{j},y^{k}_{j})}{\sigma^{k}_{j}}\eta_j^{k+1}\}_{j=1}^N, 
\notag
\end{align}
where $\eta_j^{k+1} \in \partial_1 d_j^k(x^{k+1}_{j},y^{k}_{j})$, then we have $p_{x^{k+1}} \in \partial J(x^{k+1}_{(N)})$ and
\begin{align}
      \|p_{x^{k+1}} \|\leq \rho_{b}\sum_{i=1}^N (d_i^k(x^{k+1}_{i},y^{k}_{i})+d_i^k(x^{k+1}_{i},x^{k}_{i})), \label{e35}
\end{align}   
where $\rho_{b}$ is a positive constant.  
\label{t2}
\end{theorem}

\begin{proof}
By Proposition \ref{p_metric}, Proposition \ref{p1} and Proposition \ref{p2}, Eq. (\ref{gprox}) follows that 
\begin{align}
&&0&\in \partial_{x_i}F_i(x_i^{k+1})+\nabla_{x_i}H(\{x_j^{k+1}\}_{j=1}^{i-1},y_i^k,\{x_j^k\}_{j=i+1}^{N})+\frac{1}{2\sigma_i^k}\partial_{1}(d_i^k(x_i^{k+1},y_i^k))^2 \notag \\
&& &=\partial_{x_i}F_i(x_i^{k+1})+\nabla_{x_i}H(\{x_j^{k+1}\}_{j=1}^{i-1},y_i^k,\{x_j^k\}_{j=i+1}^{N})+\frac{1}{\sigma_i^k}d_i^k(x_i^{k+1},y_i^k)\partial_{1}(d_i^k(x_i^{k+1},y_i^k)). 
\label{ec23}
\end{align}
Thus, there exists $\eta_i^{k+1}\in \partial_{1}d_i^k(x_i^{k+1},y_i^k)$ such that
\begin{align}
&&&-\nabla_{x_i}H(\{x_j^{k+1}\}_{j=1}^{i-1},y_i^k,\{x_j^k\}_{j=i+1}^{N})-\frac{d_i^k(x_i^{k+1},y_i^k)}{\sigma_i^k}\eta_i^{k+1}
\in \partial_{x_i}F_i(x_i^{k+1}).\notag
\end{align}
From that, we can get that 
\begin{align}
&&&\nabla_{x_i}H(\{x_j^{k+1}\}_{j=1}^{i-1},x_i^{k+1},\{x_j^k\}_{j=i+1}^{N})-\nabla_{x_i}H(\{x_j^{k+1}\}_{j=1}^{i-1},y_i^k,\{x_j^k\}_{j=i+1}^{N})-\frac{d_i^k(x_i^{k+1},y_i^k)}{\sigma_i^k}\eta_i^{k+1}\notag \\
&& &\in \partial_{x_i}F_i(x_i^{k+1})+\nabla_{x_i}H(\{x_j^{k+1}\}_{j=1}^{i-1},x_i^{k+1},\{x_j^k\}_{j=i+1}^{N}). \label{ec241}
\end{align}
This implies 
\begin{center}
$p_{x_{j}^{k+1}} \in \partial_{x_{j}}J(\{x_{i}^{k+1}\}_{i=1}^{j-1},x_{j}^{k+1}, \{x_{i}^{k}\}_{i=j+1}^{N}). $ 
\end{center}

From Eq. (\ref{ec23}), Assumption \ref{assump2},  Proposition \ref{p_norm_equiv}, Proposition \ref{p_sub_bound} and Lemma \ref{metric_bound}, $\exists \rho_d>0$ and $\exists L_{\nabla H}>0$ s.t.
\begin{align}
&& \|p_{x_{j}^{k+1}}\|
&=\|\nabla_{x_{j}}h_j(x^{k+1}_{j})-\nabla_{x_{j}}h_j(y^{k}_{j})-\frac{d_j^k(x_j^{k+1},y_j^k)}{\sigma_j^k}\eta_j^{k+1}\| \notag \\ 
&& &\leq \|\nabla_{x_{j}}h_j(x^{k+1}_{j})-\nabla_{x_{j}}h_j(y^{k}_{j})\|
+\|\frac{d_j^k(x_j^{k+1},y_j^k)}{\sigma_j^k}\eta_j^{k+1}\| \notag \\
&& &\leq L_{\nabla H}\|x_j^{k+1}-y_j^k\|
+\|\frac{d_j^k(x_j^{k+1},y_j^k)}{\sigma_j^k}\eta_j^{k+1}\| \notag \\
&& &\leq C_1L_{\nabla H}\|x_j^{k+1}-y_j^k\|_d
+\|\frac{d_j^k(x_j^{k+1},y_j^k)}{\sigma_j^k}\eta_j^{k+1}\| \notag \\
&& &\leq \frac{C_1L_{\nabla H}}{C_d}d_j^k(x_j^{k+1},y_j^k)
+\|\frac{d_j^k(x_j^{k+1},y_j^k)}{\sigma_j^k}\eta_j^{k+1}\| \notag \\
&& &\leq \frac{C_1L_{\nabla H}}{C_d}d_j^k(x_j^{k+1},y_j^k)
+\rho_d d_j^k(x_j^{k+1},y_j^k) \notag \\
&& &\leq \left(\frac{C_1L_{\nabla H}}{C_d}+\rho_d\right)
d_j^k(x_j^{k+1},y_j^k).
\label{ec24}
\end{align}
where $C_1>0$. Since the generated sequence is contained in a bounded level set (Assumption \ref{assump1} and Definition \ref{dlower}) and the step-size related quantities are uniformly bounded, there exists a positive constant $\rho_{b1}$ such that 
\begin{align}
&&&\frac{C_1L_{\nabla H}}{C_d}+\rho_d \leq \rho_{b1},~\forall i\in[N],~k\geq0.
\notag
\end{align}
From Eq. (\ref{ec24}), Assumption \ref{assump2}, we have
\begin{align}
&&\|\{p_{i}^{k+1} \}^{N}_{i=1}\|&\leq  \rho_{b1}\sum_{i=1}^N d_i^k(x^{k+1}_{i},y^{k}_{i}), 
\label{c32}
\end{align}
where $\rho_{b1}:=\frac{C_1L_{\nabla H}}{C_d}+\rho_d$.
Similarly, from Eq. (\ref{ec241}), we also have

\begin{align}
&&&\nabla_{x_i}H(\{x_j^{k+1}\}_{j=1}^{N})-\nabla_{x_i}H(\{x_j^{k+1}\}_{j=1}^{i-1},y_i^k,\{x_j^k\}_{j=i+1}^{N})-\frac{d_i^k(x_i^{k+1},y_i^k)}{\sigma_i^k}\eta_i^{k+1}\notag \\
&&&\in \partial_{x_i}F_i(x_i^{k+1})+\nabla_{x_i}H(\{x_j^{k+1}\}_{j=1}^{N}). \notag \\
&& &= \partial_{x_i}\sum_{j=1}^N F_j(x_j^{k+1})+\nabla_{x_i}H(\{x_j^{k+1}\}_{j=1}^{N}). \notag \\
&& &= \partial_{x_i}J(\{x_j^{k+1}\}_{j=1}^{N}). \label{ec242}
\end{align}
We also define 
\begin{align}
    &&p^{k+1}&=\{\nabla_{x_i}H(\{x_j^{k+1}\}_{j=1}^{N})-\nabla_{x_i}H(\{x_j^{k+1}\}_{j=1}^{i-1},y_i^k,\{x_j^k\}_{j=i+1}^{N})-\frac{d_i^k(x_i^{k+1},y_i^k)}{\sigma_i^k}\eta_i^{k+1}\}_{i=1}^N \notag \\
    && &\in \partial J(\{x_j^{k+1}\}_{j=1}^{N}). \label{pk}
\end{align}
 Thus, from Assumption \ref{assump1}, Eq. (\ref{ec241}), Eq. (\ref{ec242}) and Eq. (\ref{pk}), we have 
\begin{align}
\|p^{k+1}\|
\leq
\rho_b\sum_{i=1}^{N}
\left(
d_i^k(x_i^{k+1},y_i^k)+d_i^k(x_i^{k+1},x_i^k)
\right).
\notag
\end{align}
where $\rho_b$ is a positive constant. 
\end{proof}

Let $z^{k}=\left\{x^{k}_{i} \right\}^{N}_{i=1}$. Next, we prove the following lemma. 
\begin{lemma}
Let $z'$ denote the set of cluster points of $\{z^k\}_{k\in\mathbb N}$. Then $z'$ is nonempty and compact. \label{t3}
\end{lemma}

\begin{proof}
From Theorem \ref{t1}(\romannumeral1), the sequence $\{J(z^k)\}_{k\in\mathbb N}$ is monotonically decreasing. Hence,
\begin{align}
&&&J(z^k)\leq J(z^1),~\forall k\in\mathbb N. \notag
\end{align}
Therefore, the generated set $A=\{z^k:k\in\mathbb N\}$ is contained in the sublevel set
\begin{align}
&&&A\subseteq \mathcal{L}(z^1):=\{z\in\mathrm{dom}~J:J(z)\leq J(z^1)\}. \notag
\end{align}
Since $J$ is coercive by Assumption \ref{assump1}, the sublevel set $\mathcal{L}(z^1)$ is bounded. Thus $A$ is bounded. Thus, $\{z^k\}_{k\in\mathbb N}$ is bounded, $z'$ is nonempty and bounded.

Next, we prove that $z'$ is closed. Let
$\{u^m\}_{m\in\mathbb N}\subseteq z'$ satisfy
$u^m\rightarrow u$. Since $u^m\in z'$, for every
$m\in\mathbb N$, there exists an integer
$k_m>k_{m-1}$ such that
\begin{align}
\|z^{k_m}-u^m\|_d<\frac{1}{m}.
\notag
\end{align}
Therefore,
\begin{align}
\|z^{k_m}-u\|_d
&\leq
\|z^{k_m}-u^m\|_d+\|u^m-u\|_d
\longrightarrow0.
\notag
\end{align}
Hence, $u\in z'$, and thus $z'$ is closed.
Since $z'$ is bounded and closed in a finite-dimensional
normed space, it is compact. 
\end{proof}

Next, before continuing with the proof, we first prove the following lemma. 
\begin{lemma}
If $d$ is a metric and satisfies Assumption \ref{assump2}, then $d(x+z,y+z)=d(x,y)$ and $d(ax,ay)=|a|d(x,y)$ $(\forall a\in \mathbb{R},~\forall x,y,z \in \mathbb{R}^n)$. 
\label{lemma1}
\end{lemma}
\begin{proof}
$\forall i\in[N]$, $k\geq0$ and $d_i^k(\cdot,\cdot)\in\mathcal{D}_i$, let $d=d_i^k$, since $d$ is a metric and Definition \ref{d0}, we have
\begin{center}
$d(x,y)=d(y,x)$.
\end{center}
Moreover, Assumption \ref{assump2} shows that $d(\cdot,y)$ is convex for every fixed $y$. Thus, from the symmetry of $d$, $\forall \lambda\in[0,1]$, we have
\begin{align}
&&d(x,\lambda y+(1-\lambda)z)
&=d(\lambda y+(1-\lambda)z,x) \notag\\
&&&\leq\lambda d(y,x)+(1-\lambda)d(z,x) \notag\\
&&&=\lambda d(x,y)+(1-\lambda)d(x,z).
\notag
\end{align}
Therefore, $d$ is convex with respect to both arguments.

We first prove the homogeneity of $d$. For any $\lambda\in(0,1)$, from the convexity of $d$ with respect to the first argument, we have
\begin{align}
&&d(x+\lambda u,x)
&=d(\lambda(x+u)+(1-\lambda)x,x) \notag\\
&&&\leq\lambda d(x+u,x)+(1-\lambda)d(x,x) \notag\\
&&&=\lambda d(x+u,x).
\label{lemma_homo1}
\end{align}
On the other hand, from Definition \ref{d0} and the convexity of $d$ with respect to the second argument, we have
\begin{align}
&&d(x+u,x)-d(x+\lambda u,x)
&\leq d(x+u,x+\lambda u) \notag\\
&&&=d(x+u,\lambda(x+u)+(1-\lambda)x) \notag\\
&&&\leq\lambda d(x+u,x+u)+(1-\lambda)d(x+u,x) \notag\\
&&&=(1-\lambda)d(x+u,x).
\label{lemma_homo2}
\end{align}
Eq. (\ref{lemma_homo2}) implies
\begin{align}
&&&\lambda d(x+u,x)\leq d(x+\lambda u,x).
\label{lemma_homo3}
\end{align}
Thus, from Eq. (\ref{lemma_homo1}) and Eq. (\ref{lemma_homo3}), we obtain
\begin{align}
&&&d(x+\lambda u,x)=\lambda d(x+u,x),~\forall\lambda\in[0,1].
\label{lemma_homo4}
\end{align}
For $\lambda>1$, applying Eq. (\ref{lemma_homo4}) to $\frac{1}{\lambda}$ and $\lambda u$, we have
\begin{align}
&&d(x+u,x)
&=d(x+\frac{1}{\lambda}(\lambda u),x) \notag\\
&&&=\frac{1}{\lambda}d(x+\lambda u,x).
\notag
\end{align}
Therefore,
\begin{align}
&&&d(x+\lambda u,x)=\lambda d(x+u,x),~\forall\lambda\geq0.
\label{lemma_homo5}
\end{align}

Next, we prove the translation invariance of $d$. For any $\alpha>0$, from Eq. (\ref{lemma_homo5}) and Definition \ref{d0}, we have
\begin{align}
&&d(x,y)
&=\frac{1}{\alpha}d(y+\alpha(x-y),y) \notag\\
&&&\leq\frac{1}{\alpha}d(y+\alpha(x-y),y+u)
+\frac{1}{\alpha}d(y+u,y) \notag\\
&&&=d(x+u-\frac{u}{\alpha},y+u)
+\frac{1}{\alpha}d(y+u,y) \notag\\
&&&\leq d(x+u-\frac{u}{\alpha},x+u)
+d(x+u,y+u)
+\frac{1}{\alpha}d(y+u,y) \notag\\
&&&=\frac{1}{\alpha}d(x,x+u)
+d(x+u,y+u)
+\frac{1}{\alpha}d(y+u,y).
\label{lemma_trans1}
\end{align}
Letting $\alpha\rightarrow\infty$ in Eq. (\ref{lemma_trans1}), we obtain
\begin{align}
&&&d(x,y)\leq d(x+u,y+u).
\label{lemma_trans2}
\end{align}
Replacing $x$, $y$ and $u$ in Eq. (\ref{lemma_trans2}) with $x+u$, $y+u$ and $-u$, respectively, we also obtain
\begin{align}
&&&d(x+u,y+u)\leq d(x,y).
\label{lemma_trans3}
\end{align}
Therefore,
\begin{align}
&&&d(x+u,y+u)=d(x,y).
\label{lemma_trans4}
\end{align}

Finally, for any $a\geq0$, from Eq. (\ref{lemma_homo5}) and Eq. (\ref{lemma_trans4}), we have
\begin{align}
&&d(ax,ay)
&=d(ay+a(x-y),ay) \notag\\
&&&=a d(ay+x-y,ay) \notag\\
&&&=a d(x,y).
\label{lemma_homo6}
\end{align}
For any $a<0$, since $-a>0$, Eq. (\ref{lemma_homo6}), the translation invariance and the symmetry of $d$ imply
\begin{align}
&&d(ax,ay)
&=d((-a)(-x),(-a)(-y)) \notag\\
&&&=(-a)d(-x,-y) \notag\\
&&&=(-a)d(y,x) \notag\\
&&&=|a|d(x,y).
\notag
\end{align}
Hence,
\begin{align}
&&&d(ax,ay)=|a|d(x,y),~\forall a\in\mathbb{R}.\notag
\end{align}
This means Lemma \ref{lemma1} is true. 
\end{proof}

Next, we prove the following theorem.

\begin{theorem}
(i) Let $\left\{z^{k}\right\}_{k \in \mathbb{N}}$ be generated by Algorithm \ref{iGGBPL}, then $J$ is constant on $z'$. 

(ii) $z' \subseteq crit~ J$, where $crit~ J=\{x:0\in\partial J(x)\}$.


\label{t5}
\end{theorem}

\begin{proof}
(i) ${\forall}\overline{x} \in z'$, there exists a subsequence $x_{(N)}^{k_{j}}$  such that 
\begin{center}
  $\lim_{j \to \infty} x_{(N)}^{k_{j}}= \overline{x}$.  
\end{center}
Let 
\begin{align}
    && &F(x_{(N)}^{k_{j}})=\sum_{i=1}^{N}F_{i}(x_{i}^{k_{j}}). \label{ec200}
\end{align}
Since $F_{i}$ is lower semicontinuous, from Definition \ref{lower} and Eq. (\ref{ec200}), we obtain that 
\begin{align}
 && & \liminf_{j\to\infty} F(x_{(N)}^{k_{j}}) \geq F(\overline{x}). \label{t41}
\end{align}
Choosing $k=k_j-1$, from Eq. (\ref{gprox}), we infer
\begin{align}
&&F_i(x_i^{k_j})
&\leq
F_i(\overline{x_i})
+\frac{1}{2\sigma_i^{k_j-1}}
\left(
d_i^{k_j-1}
(\overline{x_i},y_i^{k_j-1})
\right)^2
\notag\\
&&&\quad
+\left\langle
R_{i,k_j-1}^{-1}
\nabla_{x_i}H
\left(
\{x_n^{k_j}\}_{n=1}^{i-1},
y_i^{k_j-1},
\{x_n^{k_j-1}\}_{n=i+1}^{N}
\right),
\overline{x_i}-x_i^{k_j}
\right\rangle_{i,k_j-1}.
\label{t421}
\end{align}
Since
$x_i^{k_j}\rightarrow\overline{x_i}$ and, by
Theorem \ref{t1}(\romannumeral2) and
Assumption \ref{assump2}(i),
\begin{align}
\|x_i^{k_j}-x_i^{k_j-1}\|_d
\longrightarrow0,
\notag
\end{align}
we have
$x_i^{k_j-1}\rightarrow\overline{x_i}$.
Moreover, since
\begin{align}
y_i^{k_j-1}
=
x_i^{k_j-1}
+\beta^{k_j-1}
\left(
x_i^{k_j-1}-x_i^{k_j-2}
\right)
\notag
\end{align}
and $\beta^{k_j-1}<1$, we obtain
\begin{align}
&&
\|\overline{x_i}-y_i^{k_j-1}\|_d
&\leq
\|\overline{x_i}-x_i^{k_j-1}\|_d
+\beta^{k_j-1}
\|x_i^{k_j-1}-x_i^{k_j-2}\|_d
\longrightarrow0.
\label{t422}
\end{align}
which also means that
$\lim_{j\to\infty}y_i^{k_j-1}=\overline{x_i}$.
From Eq. (\ref{t421}), Lemma \ref{metric_lipschitz}, and Eq. (\ref{t422}), we infer
\begin{align}
&& \limsup_{j\to\infty}F_i(x_i^{k_j})
&\leq
\limsup_{j\to\infty}
(
F_i(\overline{x_i})
+\frac{1}{2\sigma_i^{k_j-1}}
(d_i^{k_j-1}(\overline{x_i},y_i^{k_j-1}))^2 \notag \\
&& &+\Big\langle
R_{i,k_j-1}^{-1}
\nabla_{x_i}H
(
\{x_n^{k_j}\}_{n=1}^{i-1},
y_i^{k_j-1},
\{x_n^{k_j-1}\}_{n=i+1}^{N}
),
\overline{x_i}-x_i^{k_j}
\rangle_{i,k_j-1}
)
\notag\\
&&&\leq
F_i(\overline{x_i})
+\limsup_{j\to\infty}
(
\frac{1}{2\sigma_i^{k_j-1}}
(d_i^{k_j-1}(\overline{x_i},y_i^{k_j-1}))^2 \notag \\
&&&+\Big\langle
R_{i,k_j-1}^{-1}
\nabla_{x_i}H
(\{x_n^{k_j}\}_{n=1}^{i-1},
y_i^{k_j-1},
\{x_n^{k_j-1}\}_{n=i+1}^{N}),
\overline{x_i}-x_i^{k_j}
\rangle_{i,k_j-1})
\notag\\
&&&=
F_i(\overline{x_i}).
\label{ectemp}
\end{align}
 Therefore, from Eq. (\ref{ec200}) and Eq. (\ref{ectemp}), we have
\begin{align}
&& &\limsup_{j\to\infty} F(x_{(N)}^{k_{j}}) \leq \limsup_{j\to\infty} F(\overline{x}). \label{t431}
\end{align}
From Eq. (\ref{ec200}), Eq. (\ref{t431}) and Eq. (\ref{t41}), we infer 
\begin{align}
&& &\lim_{j \to \infty}F(x_{(N)}^{k_{j}}) =F(\overline{x}). 
\label{t43}
\end{align}
From Theorem \ref{t1} and Proposition \ref{tf4}, we have  
\begin{align}
  && & \lim_{j \to \infty} H(x_{(N)}^{k_{j}})=H( \overline{x}).
   \label{t432}
\end{align}
Thus from Eq. (\ref{t43}) and  Eq. (\ref{t432}), we infer 
\begin{align}
   &&\lim_{j \to \infty} H(x_{(N)}^{k_{j}})+\lim_{j \to \infty}F(x_{(N)}^{k_{j}})&=\lim_{j \to \infty} (H(x_{(N)}^{k_{j}})+F(x_{(N)}^{k_{j}}))\notag \\
   && &=\lim_{j \to \infty} J(x_{(N)}^{k_{j}}) \notag \\
   && &=J( \overline{x}). \notag
\end{align}
This means $J$ is constant on $z'$. 

(ii) From Theorem \ref{t2}, we know that
\begin{align}
&&&p_{x^{k+1}}\in \partial J(x^{k+1}_{(N)}).
\label{t21}
\end{align}
Moreover, from Theorem \ref{t1}(\romannumeral2), Assumption \ref{assump2}, Proposition \ref{p_metric}, Lemma \ref{lemma1} and Theorem \ref{t2}, we have
\begin{align}
&& \lim_{k\to\infty}\|p_{x^{k+1}}\|&\leq \lim_{k\to\infty}\rho_b\sum_{i=1}^{N}(d_i^k(x_i^{k+1},y_i^k)+d_i^k(x_i^{k+1},x_i^k)) \notag\\
&&& \leq \lim_{k\to\infty}\rho_b\sum_{i=1}^{N}d_i^k(x_i^{k+1},x_i^k)+\lim_{k\to\infty}\rho_b\sum_{i=1}^{N}d_i^k(\beta^k(x_i^{k}-x_i^{k-1}),0)  \notag\\
&&& \leq \lim_{k\to\infty}\rho_b\sum_{i=1}^{N}\omega_i(x_i^{k+1},x_i^k)+\lim_{k\to\infty}\rho_b\sum_{i=1}^{N}\omega_i(\beta^k(x_i^{k}-x_i^{k-1}),0)  \notag\\
&&& = \lim_{k\to\infty}\rho_b\sum_{i=1}^{N}\omega_i(x_i^{k+1},x_i^k)+\rho_b\sum_{i=1}^{N}\lim_{k\to\infty}\omega_i(\beta^k(x_i^{k}-x_i^{k-1}),0)  \notag\\
&&&=0.
\label{t22}
\end{align}
Let $\overline{x}\in z'$. Then there exists a subsequence $\{x_{(N)}^{k_j}\}_{j\in\mathbb N}$ such that
\begin{align}
&&&x_{(N)}^{k_j}\to\overline{x}.
\label{t44}
\end{align}
By Theorem \ref{t5}(\romannumeral1), we have
\begin{align}
&&&J(x_{(N)}^{k_j})\to J(\overline{x}).
\label{t45}
\end{align}
Taking the subsequence in Eq. (\ref{t21}) and Eq. (\ref{t22}), we obtain
\begin{align}
&&&p_{x^{k_j}}\in\partial J(x_{(N)}^{k_j}),\quad p_{x^{k_j}}\to0.
\label{t46}
\end{align}
Therefore, from the fact that the limiting subdifferential has a sequentially closed graph, Eq. (\ref{t44}), Eq. (\ref{t45}) and Eq. (\ref{t46}) imply 
\begin{align}
&&&0\in\partial J(\overline{x}). \notag
\end{align}
\end{proof}

\subsection{Global convergence}
Before continuing with the proof of the Lemma and Theorems in the manuscript, we prove the following lemma. 
\begin{lemma}
\label{metric_lipschitz}
Suppose that Assumption \ref{assump2}(ii)--(iii) holds. For every $i\in[N]$ and every compact set $K_i\subseteq\mathbb{R}^{d_i}$, there exists a positive constant $L_{i,K_i}$ s.t.
\begin{align}
|d(x,y)-d(z,y)|
\leq
L_{i,K_i}\|x-z\|_d,
\quad
\forall d\in\mathcal{D}_i,\quad
\forall x,z,y\in K_i.
\notag
\end{align}
Moreover,
\begin{align}
\|\eta\|_{d,*}
\leq
L_{i,K_i},
\quad
\forall d\in\mathcal{D}_i,\quad
\forall x,y\in K_i,\quad
\forall\eta\in\partial_1d(x,y),
\notag
\end{align}
where $\|\cdot\|_{d,*}$ denotes the dual norm of $\|\cdot\|_d$.
\end{lemma}

\begin{proof}
Fix $i\in[N]$ and a compact set $K_i\subseteq\mathbb{R}^{d_i}$. Since $K_i$ is bounded, there exists $R_i>0$ s.t.
\begin{align}
K_i
\subseteq
\{x\in\mathbb{R}^{d_i}:\|x\|_d\leq R_i\}.
\notag
\end{align}
Define
\begin{align}
\widehat{K}_i
:=
\{u\in\mathbb{R}^{d_i}:\|u\|_d\leq R_i+1\}.
\notag
\end{align}
The set $\widehat{K}_i\times K_i$ is compact. By Assumption \ref{assump2}(iii),
\begin{align}
M_{i,K_i}
:=
\sup_{\substack{u\in\widehat{K}_i~y\in K_i}}
\omega_i(u,y)
=
\sup_{\substack{d\in\mathcal{D}_i\\u\in\widehat{K}_i,~y\in K_i}}
d(u,y)
<
+\infty.
\notag
\end{align}

Let $d\in\mathcal{D}_i$, $x,y\in K_i$, and $\eta\in\partial_1d(x,y)$. For every $v\in\mathbb{R}^{d_i}$ satisfying $\|v\|_d\leq1$, Assumption \ref{assump2}(ii) gives
\begin{align}
d(x+v,y)
&\geq
d(x,y)+\langle\eta,v\rangle,
\notag\\
d(x-v,y)
&\geq
d(x,y)-\langle\eta,v\rangle.
\notag
\end{align}
Since $x+v,x-v\in\widehat{K}_i$, the definition of $M_{i,K_i}$ gives
\begin{align}
d(x+v,y)
\leq
M_{i,K_i},
\qquad
d(x-v,y)
\leq
M_{i,K_i}.
\notag
\end{align}
Moreover, $d(x,y)\geq0$. Therefore,
\begin{align}
-M_{i,K_i}
\leq
\langle\eta,v\rangle
\leq
M_{i,K_i},
\quad
\forall\|v\|_d\leq1.
\notag
\end{align}
By the definition of the dual norm,
\begin{align}
\|\eta\|_{d,*}
=
\sup_{\|v\|_d\leq1}
|\langle\eta,v\rangle|
\leq
M_{i,K_i}.
\notag
\end{align}

Since $d(\cdot,y)$ is finite and convex on $\mathbb{R}^{d_i}$, its subdifferential is nonempty at every point. For any $x,z,y\in K_i$, take $\eta_x\in\partial_1d(x,y)$ and $\eta_z\in\partial_1d(z,y)$. The corresponding subgradient inequalities give
\begin{align}
d(z,y)-d(x,y)
&\geq
\langle\eta_x,z-x\rangle,
\notag\\
d(z,y)-d(x,y)
&\leq
\langle\eta_z,z-x\rangle.
\notag
\end{align}
Using the uniform subgradient bound, we obtain
\begin{align}
-M_{i,K_i}\|x-z\|_d
&\leq
d(z,y)-d(x,y)
\notag\\
&\leq
M_{i,K_i}\|x-z\|_d.
\notag
\end{align}
Hence,
\begin{align}
|d(x,y)-d(z,y)|
\leq
M_{i,K_i}\|x-z\|_d.
\notag
\end{align}
The results follow by setting $L_{i,K_i}=M_{i,K_i}$.
\end{proof}

Next, we prove Property \ref{based}(\romannumeral4) that the sequence generated by our methods has global convergence and quantify the convergence rate of the proposed methods. 
\begin{theorem}
The sequence $\{z^{k}\}_{k \in \mathbb{N}}$ generated by Algorithm \ref{iGGBPL} converges, i.e., $\lim_{k\rightarrow \infty} \sup_{p\in \mathbb{N}}\|z^{k+p}-z^{k}\|_d=0$ and the finite-dimensional normed space induced by $\|\cdot\|_d$ is complete. 
\label{glo}
\end{theorem}

\begin{proof}
Let $\overline{z}\in z'$. 
If $J(z^k)=J(\overline{z})$ for some $k$, then Theorem \ref{t1}(\romannumeral1) implies that
$d_i^k(x_i^{k+1},x_i^k)=0$ for all $i\in[N]$, and the sequence remains constant thereafter.
Therefore, we assume that
$J(z^k)>J(\overline{z})$ for all $k$.

Since $z'$ is compact and $J$ is constant on $z'$ by Theorem \ref{t5}, 
and $\operatorname{dist}(z^k,z')\rightarrow0$,
the uniformized KŁ property (Proposition \ref{def_uniKL}) yields a concave function $\phi$ and an integer $l\geq1$ such that
\begin{align}
\phi^{'}(J(z^{k})-J(\overline{z}))
\operatorname{dist}(0,\partial J(z^{k}))
\geq1,
\quad \forall k\geq l.
\label{KL1}
\end{align}
Since $\phi$ is a concave function, we have
\begin{align}
&&\phi(J(z^{k+1})-J( \overline{z}))&\leq \phi(J(z^{k})-J( \overline{z}))+\phi^{'}(J(z^{k})-J( \overline{z}))(J(z^{k+1})-J(z^{k})). 
\label{temp1}
\end{align}
From Theorem \ref{t2}, we infer 
\begin{align}
&& \operatorname{dist}(0, \partial J(z^{k}) ) &\leq \|\left\{p_{x^{k}_{i}} \right\}^{N}_{i=1}\| \notag \\
&& &\leq \rho_{b}\sum_{i=1}^{N}(d_i^{k-1}(x_i^{k},y_i^{k-1})+d_i^{k-1}(x_i^{k},x_i^{k-1})). \label{glo1}
\end{align}
Since the Eq. (\ref{KL1}) and Eq. (\ref{glo1}), we have 
\begin{align}
&&  \phi^{'}(J(z^{k})-J(\overline{z})) &\geq \frac{1}{\operatorname{dist}(0, \partial J(z^{k}))} \notag \\
&& &\geq \frac{1}{\rho_{b}\sum_{i=1}^{N}(d_i^{k-1}(x_i^{k},y_i^{k-1})+d_i^{k-1}(x_i^{k},x_i^{k-1}))}. \label{glo2}
\end{align}
Let $G(k)=J(z^{k})-J( \overline{z})$, from  Eq. (\ref{temp1}), Eq. (\ref{glo1}) and Eq. (\ref{glo2}), we have  
\begin{align}
&& \phi(G(k))-\phi(G(k+1)) &\geq \phi^{'}(G(k))(G(k)-G(k+1))  \notag \\
&& &\geq \frac{G(k)-G(k+1)}{\rho_{b}\sum_{i=1}^{N}(d_i^{k-1}(x_i^{k},y_i^{k-1})+d_i^{k-1}(x_i^{k},x_i^{k-1}))} \notag \\
&& &\geq \frac{\rho\sum_{i=1}^{N}(d_i^k(x_i^{k+1},x_i^k))^2}{\rho_{b}\sum_{i=1}^{N}(d_i^{k-1}(x_i^{k},x_i^{k-1})+d_i^{k-1}(x_i^{k},y_i^{k-1}))}. \notag 
\end{align}
Thus, we infer 
  \begin{align}
    \rho\sum_{i=1}^{N}(d_i^k(x_i^{k+1},x_i^k))^2\leq  (\phi(G(k))-\phi(G(k+1))) \rho_{b}\sum_{i=1}^{N}(d_i^{k-1}(x_i^{k},x_i^{k-1})+d_i^{k-1}(x_i^{k},y_i^{k-1}))\label{glotemp1}.
    \end{align}
From Definition \ref{d0}, Definition \ref{dlower}, Lemma \ref{lemma1}, Lemma \ref{metric_lipschitz} and Proposition \ref{p_norm_equiv}, in this bounded subset, there exists a positive constant $C_D$ such that
\begin{align}
&&d_i^{k-1}(x_i^{k},x_i^{k-1})+d_i^{k-1}(x_i^{k},y_i^{k-1})&\leq 2d_i^{k-1}(x_i^{k},x_i^{k-1})+d_i^{k-1}(\beta^{k-1}(x_i^{k-1}-x_i^{k-2}),0) \notag \\
&&&\leq 2(d_i^{k-1}(x_i^{k},x_i^{k-1})+d_i^{k-1}(\beta^{k-1}(x_i^{k-1}-x_i^{k-2}),0)) \notag \\
&& &\leq C_D(\|x_i^{k}-x_i^{k-1}\|_d+\beta^{k-1}\|x_i^{k-1}-x_i^{k-2}\|_d) \notag \\
\end{align}
Therefore, from the above formulation, $\beta^k<1$, Assumption \ref{assump2} and Eq. (\ref{glotemp1}), define $T_k=\sum_{i=1}^{N}\|x_i^{k+1}-x_i^k\|_d$ and $C_{a}=\frac{N\rho_b}{\rho C_d^2}$, we infer
\begin{align}
T_k^2
&\leq
\frac{N}{C_d^2}
\sum_{i=1}^{N}
(d_i^k(x_i^{k+1},x_i^k))^2
\notag\\
&\leq
\frac{N\rho_b}{\rho C_d^2}
(\phi(G(k))-\phi(G(k+1)))
\sum_{i=1}^{N}(d_i^{k-1}(x_i^k,y_i^{k-1})+d_i^{k-1}(x_i^k,x_i^{k-1}))
\notag\\
&\leq
C_aC_D
(\phi(G(k))-\phi(G(k+1)))
(T_{k-1}+T_{k-2}),
\notag
\end{align}
Let $C=2C_{a}C_D$, then using the fact that $2ab\leq a^{2}+b^{2}$ 
  \begin{center}
    $2 T_k \leq C (\phi(G(k))-\phi(G(k+1))) +\frac{1}{2}(T_{k-1}+T_{k-2})$.
    \end{center}
Sum both sides 
\begin{align}
&&  2 \sum_{k=l+1}^{K} T_k &\leq \frac12\sum_{k=l+1}^{K} T_{k-1}+\frac12\sum_{k=l+1}^{K}T_{k-2}+C(\phi(G(l+1))-\phi(G(K+1))) \notag \\
&& &\leq C (\phi(G(l+1))-\phi(G(K+1)))+T_{l}+T_{l-1}+\sum_{k=l+1}^{K}T_k. \label{glotemp0}
\end{align}
From Assumption \ref{assump1} and Eq. (\ref{glotemp0}),  we can get that 
\begin{align}
&& \lim_{K\rightarrow \infty}\sum_{k=l+1}^{K} T_k &\leq T_{l}+T_{l-1}+ C\phi(G(l+1))- \lim_{K\rightarrow \infty} C\phi(G(K+1)) \notag \\
&& &= C \phi(G(l+1))+T_{l}+T_{l-1}-C\phi(\lim_{K \to \infty} G(K+1))  \notag \\
 && & < \infty .\notag
\end{align}
Thus we have 
\begin{align}
&& \sum_{k=l+1}^{\infty} T_k &< \infty. \label{c7}
\end{align}
From Assumption \ref{assump2} and Eq. (\ref{c7}), it shows that 
\begin{align}
&& \lim_{K\rightarrow \infty} \sup_{p\in \mathbb{N}}\|z^{K+p}-z^{K}\|_d& \leq \lim_{K\rightarrow \infty}\sum_{k=K}^{\infty}\|z^{k+1}-z^{k}\|_d \notag\\
&& &\leq\lim_{K\rightarrow \infty}\sum_{k=K}^{\infty} T_k \notag\\
&& &= 0.  \notag
\end{align}
Given that $\|\cdot\|_d$ is a norm and Proposition \ref{p_norm_equiv}, the above formulation means that Theorem \ref{glo} is true.

\end{proof}

\begin{theorem} 
(Convergence rate) Let $\{z^{k}\}_{k\in\mathbb{N}}$ be the sequence generated by Algorithm \ref{iGGBPL}. Suppose that the desingularizing function has the form $\phi(t)=\frac{C}{\theta}t^{\theta}$, where $\theta\in(0,1]$ and $C>0$. Let $r^k=|J(z^k)-J^*|$, where $J^*=\lim_{k\rightarrow\infty}J(z^k)$. The following assertions hold.

(\romannumeral1) If $\theta=1$, Algorithm \ref{iGGBPL} terminates in finite steps.

(\romannumeral2) If $\theta\in[\frac{1}{2},1)$, then there exist positive constants $c_1$, $q\in(0,1)$, and an integer $k_2$ such that
$r^k
\leq
c_1q^{k-k_2},
\quad
\forall k\geq k_2.$

(\romannumeral3) If $\theta\in(0,\frac{1}{2})$, then there exist a positive constant $c_2$ and an integer $k_3$ such that
$
r^k
\leq
c_2(k-k_3)^{-\frac{1}{1-2\theta}},
\quad
\forall k>k_3.
$
\label{rate}
\end{theorem}

\begin{proof}
If $r^k=0$ for some $k\in\mathbb{N}$, the result follows directly from Theorem \ref{t1}. Therefore, we assume that $r^k>0$ for every $k\in\mathbb{N}$. From Proposition \ref{def_uniKL}, there exists an integer $k_0$ such that
\begin{align}
\phi^{\prime}(r^k)
\operatorname{dist}(0,\partial J(x^{k}_{(N)}))
\geq
1,
\quad
\forall k\geq k_0.
\label{rate_KL}
\end{align}

Define
\begin{align}
A_k
:=
\sum_{i=1}^{N}
\left(
d_i^{k-1}(x_i^k,x_i^{k-1})
\right)^2.
\notag
\end{align}
Since the generated sequence is bounded, Lemma \ref{metric_lipschitz}, Assumption \ref{assump2}(i), and
$y_i^{k-1}=x_i^{k-1}+\beta^{k-1}(x_i^{k-1}-x_i^{k-2})$
show that there exists a positive constant $C_y$ such that
\begin{align}
d_i^{k-1}(x_i^k,y_i^{k-1})
&\leq
C_y\|x_i^k-y_i^{k-1}\|_d
\notag\\
&\leq
C_y\left(
\|x_i^k-x_i^{k-1}\|_d
+
\beta^{k-1}
\|x_i^{k-1}-x_i^{k-2}\|_d
\right)
\notag\\
&\leq
\frac{C_y}{C_d}
\left(
d_i^{k-1}(x_i^k,x_i^{k-1})
+
d_i^{k-2}(x_i^{k-1},x_i^{k-2})
\right).
\label{rate_metric}
\end{align}
Therefore, from Theorem \ref{t2}, Eq. (\ref{rate_metric}), and the Cauchy--Schwarz inequality, there exists a positive constant $C_b$ such that
\begin{align}
dist^2(0,\partial J(x^{k}_{(N)}))
\leq
C_b(A_k+A_{k-1}).
\label{rate_subgradient}
\end{align}
Moreover, Theorem \ref{t1} gives
\begin{align}
\rho A_k
&\leq
r^{k-1}-r^k,
\notag\\
\rho A_{k-1}
&\leq
r^{k-2}-r^{k-1}.
\notag
\end{align}
Combining these inequalities with Eq. (\ref{rate_KL}), we obtain
\begin{align}
1
&\leq
\left(
\phi^{\prime}(r^k)
\operatorname{dist}(0,\partial J(x^{k}_{(N)}))
\right)^2
\notag\\
&\leq
\frac{C_b}{\rho}
\left(
\phi^{\prime}(r^k)
\right)^2
(r^{k-2}-r^k).
\label{eq_rate1}
\end{align}
Let $d_1=C_b/\rho$. Since
$\phi(t)=\frac{C}{\theta}t^{\theta}$, we have
$\phi^{\prime}(t)=Ct^{\theta-1}$. Hence,
\begin{align}
1
\leq
d_1C^2
(r^k)^{2\theta-2}
(r^{k-2}-r^k).
\label{eq_rate2}
\end{align}

We next discuss the three possible cases of $\theta$.

(\romannumeral1) Let $\theta=1$. From Eq. (\ref{eq_rate2}), we obtain
\begin{align}
r^{k-2}-r^k
\geq
\frac{1}{d_1C^2},
\quad
\forall k\geq k_0.
\notag
\end{align}
If $r^k>0$ for infinitely many $k$, each of the two nonnegative subsequences $\{r^{2k}\}$ and $\{r^{2k+1}\}$ decreases by at least the fixed positive constant $1/(d_1C^2)$ at every step. This is impossible. Therefore, $r^k=0$ for some finite $k$. Theorem \ref{t1} then implies that all subsequent block increments vanish, and Algorithm \ref{iGGBPL} terminates in finite steps.

(\romannumeral2) Let $\theta\in[\frac{1}{2},1)$. Since $r^k\rightarrow0$, there exists an integer $k_2\geq k_0$ such that
\begin{align}
0<r^k\leq1,
\quad
\forall k\geq k_2.
\notag
\end{align}
Since $2-2\theta\in(0,1]$, we have
$r^k\leq(r^k)^{2-2\theta}$. Therefore, Eq. (\ref{eq_rate2}) gives
\begin{align}
r^k
&\leq
(r^k)^{2-2\theta}
\notag\\
&\leq
d_1C^2(r^{k-2}-r^k),
\quad
\forall k\geq k_2.
\notag
\end{align}
Thus,
\begin{align}
r^k
\leq
q_0r^{k-2},
\qquad
q_0
:=
\frac{d_1C^2}{1+d_1C^2}
\in(0,1).
\label{rate_linear}
\end{align}
Applying Eq. (\ref{rate_linear}) separately to the even and odd subsequences, we obtain
\begin{align}
r^k
\leq
q_0^{\lfloor(k-k_2)/2\rfloor}
r^{k_2},
\quad
\forall k\geq k_2.
\notag
\end{align}
Let $q=\sqrt{q_0}$ and $c_1=r^{k_2}/q$. Since
$q_0^{\lfloor(k-k_2)/2\rfloor}
\leq
q^{-1}q^{k-k_2}$, we have
\begin{align}
r^k
\leq
c_1q^{k-k_2},
\quad
\forall k\geq k_2.
\notag
\end{align}

(\romannumeral3) Let $\theta\in(0,\frac{1}{2})$ and define
\begin{align}
\delta
:=
1-2\theta
>
0,
\qquad
a
:=
\frac{1}{d_1C^2}.
\notag
\end{align}
Eq. (\ref{eq_rate2}) can be equivalently written as
\begin{align}
r^{k-2}-r^k
\geq
a(r^k)^{1+\delta}.
\label{rate_sublinear}
\end{align}
For $\ell\in\{0,1\}$, define
\begin{align}
s_m^{(\ell)}
:=
r^{k_0+2m+\ell}.
\notag
\end{align}
From Eq. (\ref{rate_sublinear}), we have
\begin{align}
s_{m-1}^{(\ell)}-s_m^{(\ell)}
\geq
a\left(s_m^{(\ell)}\right)^{1+\delta},
\quad
\forall m\geq1.
\label{rate_sequence}
\end{align}

We consider two cases. If
$s_m^{(\ell)}\leq\frac{1}{2}s_{m-1}^{(\ell)}$, then
\begin{align}
\left(s_m^{(\ell)}\right)^{-\delta}
-
\left(s_{m-1}^{(\ell)}\right)^{-\delta}
&\geq
(2^\delta-1)
\left(s_{m-1}^{(\ell)}\right)^{-\delta}
\notag\\
&\geq
(2^\delta-1)
\left(s_0^{(\ell)}\right)^{-\delta}.
\label{rate_case1}
\end{align}
If
$s_m^{(\ell)}>\frac{1}{2}s_{m-1}^{(\ell)}$, the mean value theorem and Eq. (\ref{rate_sequence}) give
\begin{align}
\left(s_m^{(\ell)}\right)^{-\delta}
-
\left(s_{m-1}^{(\ell)}\right)^{-\delta}
&\geq
\delta
\left(s_{m-1}^{(\ell)}\right)^{-1-\delta}
\left(
s_{m-1}^{(\ell)}-s_m^{(\ell)}
\right)
\notag\\
&\geq
\delta a
\left(s_{m-1}^{(\ell)}\right)^{-1-\delta}
\left(s_m^{(\ell)}\right)^{1+\delta}
\notag\\
&\geq
\frac{\delta a}{2^{1+\delta}}.
\label{rate_case2}
\end{align}
Define
\begin{align}
c_\ell
:=
\min
\left\{
(2^\delta-1)
\left(s_0^{(\ell)}\right)^{-\delta},
\frac{\delta a}{2^{1+\delta}}
\right\}
>
0.
\notag
\end{align}
From Eq. (\ref{rate_case1}) and Eq. (\ref{rate_case2}), we obtain
\begin{align}
\left(s_m^{(\ell)}\right)^{-\delta}
\geq
\left(s_0^{(\ell)}\right)^{-\delta}
+
c_\ell m.
\notag
\end{align}
Consequently,
\begin{align}
s_m^{(\ell)}
\leq
(c_\ell m)^{-1/\delta},
\quad
\forall m\geq1.
\notag
\end{align}
Let $c_0^{\prime}=\min\{c_0,c_1\}$. Since
$\lfloor(k-k_0)/2\rfloor\geq(k-k_0)/3$ for every $k\geq k_0+2$, we obtain
\begin{align}
r^k
\leq
\left(
\frac{c_0^{\prime}}{3}
\right)^{-1/\delta}
(k-k_0)^{-1/\delta},
\quad
\forall k\geq k_0+2.
\notag
\end{align}
The result follows by setting
\begin{align}
k_3
:=
k_0+1,
\qquad
c_2
:=
\left(
\frac{c_0^{\prime}}{3}
\right)^{-\frac{1}{1-2\theta}}.
\notag
\end{align}
This completes the proof.
\end{proof}

\subsection{Iteration complexity}
As an immediate consequence of the sufficient decrease in Theorem \ref{t1} and the subgradient bound in Theorem \ref{t2}, we further establish a global finite-iteration complexity guarantee under the arbitrary inner products and general admissible metrics considered in this paper.
\begin{theorem}
(Iteration complexity) Let $\{z^k\}_{k\in\mathbb{N}}$ be the sequence generated by Algorithm \ref{iGGBPL}, and let $J^*=\lim_{k\rightarrow\infty}J(z^k)$. There exists a positive constant $C_{\mathrm{it}}$, independent of $K$, such that, for every integer $K\geq1$,
\begin{align}
\min_{1\leq k\leq K}\operatorname{dist}^{2}(0,\partial J(z^{k+1}))
\leq
\frac{C_{\mathrm{it}}(J(z^1)-J^*)}{K}.
\label{complexity_bound}
\end{align}
Thus, the minimum squared stationarity residual within the first $K$ iterations is bounded by $\mathcal{O}(K^{-1})$. Consequently, under the criterion $\operatorname{dist}^{2}(0,\partial J(z^{k+1}))\leq\varepsilon$, the iteration complexity is $\mathcal{O}(\varepsilon^{-1})$. Equivalently, under the criterion $\operatorname{dist}(0,\partial J(z^{k+1}))\leq\varepsilon$, the iteration complexity is $\mathcal{O}(\varepsilon^{-2})$. The same assertions hold for Algorithm \ref{GGBPL}.
\label{complexity}
\end{theorem}
\begin{proof}
Define
\begin{align}
A_k
:=
\sum_{i=1}^{N}
\left(
d_i^{k-1}(x_i^k,x_i^{k-1})
\right)^2.
\notag
\end{align}
From Eq. (\ref{rate_subgradient}), for every $k\geq1$, we have
\begin{align}
\operatorname{dist}^{2}(0,\partial J(z^{k+1}))
\leq
C_b(A_{k+1}+A_k).
\notag
\end{align}
Since $z^1=z^0$, we have $A_1=0$. Therefore, for every integer $K\geq1$,
\begin{align}
\sum_{k=1}^{K}
\operatorname{dist}^{2}(0,\partial J(z^{k+1}))
&\leq
2C_b\sum_{k=2}^{K+1}A_k
\notag\\
&\leq
\frac{2C_b}{\rho}
\left(J(z^1)-J(z^{K+1})\right)
\notag\\
&\leq
\frac{2C_b}{\rho}
\left(J(z^1)-J^*\right).
\notag
\end{align}
Dividing both sides by $K$ and setting $C_{\mathrm{it}}:=2C_b/\rho$ give Eq. (\ref{complexity_bound}).
\end{proof}

\section{Numerical experiments} 
\label{Experi}
In this section, we evaluate the performance of the proposed GGBPL and iGGBPL methods on both sparse nonnegative matrix factorization with $\ell_0$ constraints and sparse nonnegative CP decomposition with $\ell_0$ constraints problems. 
The experiments are carried out in MATLAB 2021b on a computer with an Intel(R) Core(TM) i7-10850H CPU @ 2.70GHz and 16.0 GB RAM. 
For tensor computations, we use Tensor Toolbox 3.5 \cite{bader2006algorithm}. 
Our code is publicly available at \url{https://github.com/Weifeng-Yang/GGBPL}.


\subsection{Baseline Methods}
\label{subsection: asup}

To verify the numerical effectiveness of the proposed methods, we compare them with the following methods for solving Eq. (\ref{e11}).

\begin{itemize} 
    
    \item [1)] PALM  \cite{bolte2014proximal}:  The Proximal Alternating Linearized Minimization  method. 
    
    \item [2)] BPL  \cite{xu2017globally}:  Randomized/deterministic block prox-linear  method. 

    \item [3)] IBPG  \cite{le2020inertial}:  Inertial block proximal gradient method.  

    \item [4)] TITAN  \cite{phan2023inertial}: inerTIal block majorizaTion minimizAtioN method.  

    \item [5)] ABPL  \cite{yang2023accelerated}:  Randomized/deterministic accelerated block proximal linear method with adaptive momentum.

    \item [6)] PGels \cite{yang2024proximal}: Proximal gradient method with extrapolation and line search.

    \item [7)] iPAL \cite{yang2025sparse}: inertial Proximal Alternating Linearized method.  This method provides a more flexible version of the ABPL method that simplifies the two restart steps required in ABPL to a single restart step.

\end{itemize}
All parameters of the compared methods are set according to the corresponding reference papers. 
Although several ADMM algorithms \cite{wang2019global,li2025proximal,bai2026proximal} and some nonsmooth optimization algorithms \cite{jiang2025inexact,lyaqini2024non,xiao2024adam} can also handle some multiblock nonconvex optimization problems, these problem formulations are different from Eq. (\ref{e11}), smoothing methods even only yield approximate solutions. 
Therefore, these ADMM algorithms and nonsmooth optimization algorithms are out of the scope of this paper.

\subsection{Sparse nonnegative matrix factorization with $\ell_0$-constraints ($\ell_0$-SNMF)} 
\subsubsection{$\ell_0$-SNMF model}

Nonnegative matrix factorization (NMF) and its sparse variants are popular and powerful tools for feature extraction \cite{li2025superpixel}.   
To obtain sparse and interpretable solutions, sparsity constraints are usually imposed on the factor matrices, and the $\ell_0$-norm constraint gives a more explicit sparsity description since it directly controls the number of nonzero entries \cite{weston2003use,shi2022cardinality}. 
However, optimizing the sparse NMF with $\ell_0$-norm constraints ($\ell_0$-SNMF) is known to be NP-hard, nonconvex and nonsmooth, and the factor form of NMF allows the metric used for one factor to be constructed from the other factor. 
Therefore, instead of using the standard inner product and its induced Euclidean norm, we construct a variable-geometry form of the proposed methods by introducing nonstandard inner products and metrics derived from the local geometric structure of $\ell_0$-SNMF. 
This construction enables each block update to use the local scaling information induced by the corresponding block variables, and leads to a practical convergent scheme for solving $\ell_0$-SNMF.

Mathematically, $\ell_0$-SNMF can be presented as follows. 
\begin{equation}
\begin{aligned}
&&  ~\min \limits_{U,V} &\frac{1}{2}\|X-UV\|_{F}^{2}+\frac{\alpha_1}{2}\|U\|_F^2+\frac{\alpha_2}{2}\|V\|_F^2,\\
&&& s. t. ~  U,V \geq 0,~\|U\|_{0} \leq s_{1},~\|V\|_{0} \leq s_{2},   
\end{aligned}
\label{e41}
\end{equation}
where $X \in \mathbb{R}^{m\times n}$, $U\in\mathbb{R}^{m\times r}$ and $V\in\mathbb{R}^{r\times n}$.

\subsubsection{Solving $\ell_0$-SNMF using iGGBPL}
If we write Eq. (\ref{e41}) in the form of Eq. (\ref{e11}), then we have 
\begin{align}
&&H(U,V)&=\frac{1}{2}\|X-UV\|_{F}^{2}+\frac{\alpha_1}{2}\|U\|_F^2+\frac{\alpha_2}{2}\|V\|_F^2, \notag \\
&& F_{1}(U)&=\delta_{1}(U),~  F_{2}(V)=\delta_{2}(V),\label{smNMFEq}
\end{align}
where
\begin{equation}
\delta_{i}(x)=\left\{
\begin{aligned}
0&, \quad x\geq 0,~\|x\|_{0} \leq s_{i}, \\
\infty&,\quad else.
\end{aligned}
\right. 
\label{delta}
\end{equation}
It is easy to see that $\nabla_{U} H(U,V)$ and $\nabla_{V} H(U,V)$ are 
\begin{align}
\nabla_{U} H(U,V)=(UV-X)V^{T}+\alpha_1U, \nabla_{V} H(U,V)=U^{T}(UV-X)+\alpha_2V.\notag
\end{align}
From the above formulation, it is easy to see that the local geometric structure of $H(U,V)$ with respect to $U$ is characterized by $VV^T+\alpha_1I$.
Similarly, the local geometric structure with respect to $V$ is characterized by $U^TU+\alpha_2I$.
Therefore, instead of using the standard inner product and its induced Euclidean norm, we construct metrics for block variables based on this local geometric information as follows. 
\begin{align}
M_U^k&=\operatorname{Diag}\left(V^k(V^k)^T\mathbf{1}\right)+(\alpha_1+\varepsilon)I, \notag\\
M_V^k&=\operatorname{Diag}\left((U^{k+1})^TU^{k+1}\mathbf{1}\right)+(\alpha_2+\varepsilon)I, \notag
\end{align}
where $\varepsilon>0$ and $\mathbf{1}$ denotes the all-one vector with a proper dimension. 
Then the nonstandard inner products and their induced metrics are defined as
\begin{align}
&&&\langle A,B\rangle_U^k=\operatorname{tr}(A M_U^k B^T),~\langle A,B\rangle_V^k=\operatorname{tr}(A^T M_V^k B), \notag\\
&&&d_U^k(A,B)=\|A-B\|_{M_U^k}=\sqrt{\operatorname{tr}((A-B)M_U^k(A-B)^T)}, \notag\\
&&&d_V^k(A,B)=\|A-B\|_{M_V^k}=\sqrt{\operatorname{tr}((A-B)^TM_V^k(A-B))}. \notag
\end{align}
Since $M_U^k$ and $M_V^k$ are derived from the local geometric structures characterized by $V^k(V^k)^T+\alpha_1I$ and $(U^{k+1})^TU^{k+1}+\alpha_2I$, the block update scheme of the proposed methods can utilize the local scaling information of the corresponding subproblem of the $\ell_0$-SNMF model.

Moreover,  the Frobenius norm is obviously a KŁ function, Eq. (\ref{delta}) is a proper and lower semicontinuous KŁ function \cite{attouch2013convergence}. Thus, Eq. (\ref{smNMFEq}) satisfies Assumption \ref{assump1}, $d_U^k(A,B)$ and $d_V^k(A,B)$ also obviously satisfy Assumption \ref{assump2}. 
By applying the generalized geometry proximal linearized operator (Eq. (\ref{gprox})) to Eq. (\ref{e41}), the update scheme of the proposed methods for solving the $\ell_0$-SNMF can be expressed as follows. 
\begin{align} 
 && U^{k+1} &\in Gprox_{\sigma_{1}^{k} F_{1}}^{d_U^k}\left(Y_U^k;\nabla_UH(Y_U^k,V^k)(M_U^k)^{-1}\right) \notag \\
& & &=\operatorname*{\arg\min} \limits_{A}
 \left\{\|A-B_U^k\|_{M_U^k}^{2}: A\geq 0,~ \|A\|_{0}\leq s_{1}\right\}, \notag  \\
 && V^{k+1} &\in Gprox_{\sigma_{2}^{k} F_{2}}^{d_V^k}\left(Y_V^k;(M_V^k)^{-1}\nabla_VH(U^{k+1},Y_V^k)\right) \notag \\
&& &= \operatorname*{\arg\min} \limits_{A}
 \left\{\|A-B_V^k\|_{M_V^k}^{2}: A\geq 0,~ \|A\|_{0}\leq s_{2}\right\}, \notag  
\end{align}
where $\sigma_{1}^{k}=\frac{1}{\gamma_{1}^{k}},~
\sigma_{2}^{k}=\frac{1}{\gamma_{2}^{k}},~
\gamma_{1}^{k}>1,~\gamma_{2}^{k}>1$ and 
\begin{align}
&& B_U^k&=Y_U^k-\sigma_1^k
\nabla_UH(Y_U^k,V^k)(M_U^k)^{-1},\\
&& B_V^k&=Y_V^k-\sigma_2^k(M_V^k)^{-1}\nabla_VH(U^{k+1},Y_V^k),\notag 
\end{align}


\begin{figure}[!h]
    \centering 
    \includegraphics[width=0.48\linewidth]{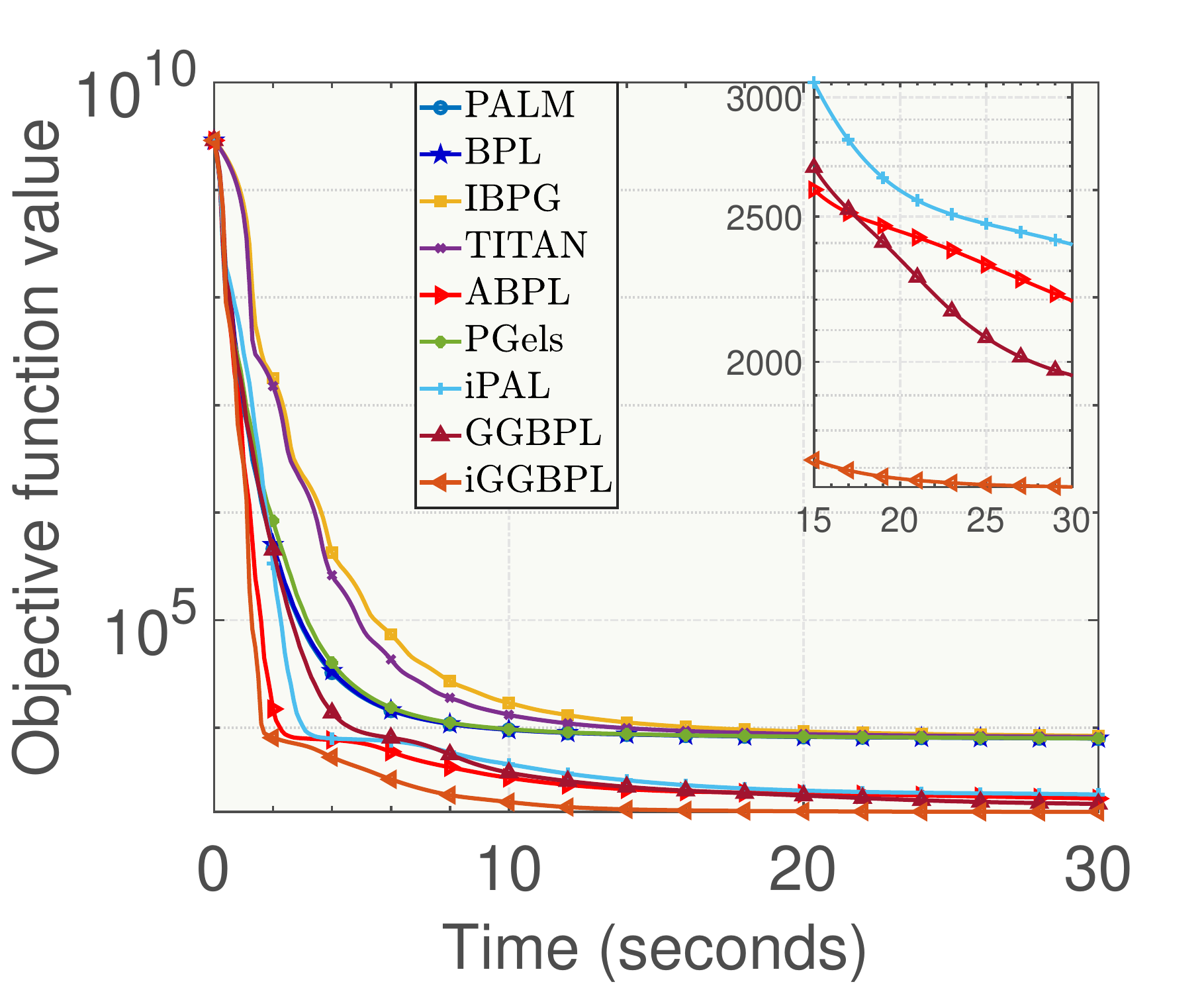}
    \includegraphics[width=0.48\linewidth]{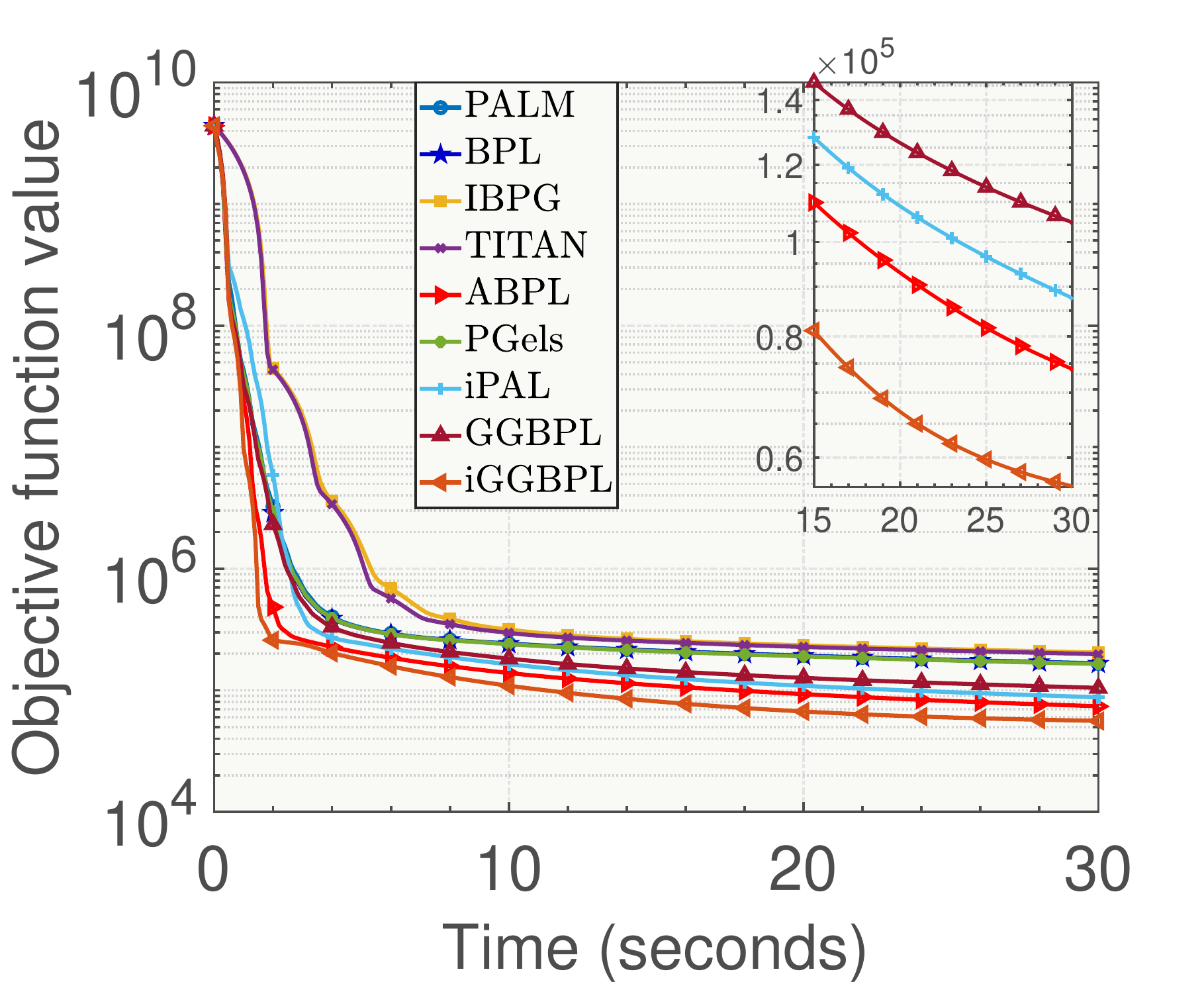}
    
    \caption{
    Average convergence behavior on the lp\_ship12l (left) and  BASEHOCK (right) datasets. The inset enlarges the final stage objective function values of the four best-performing methods. 
    }\label{figmat}
    
\end{figure}

\begin{figure*}[!h]
    \centering 
    \includegraphics[width=0.44\linewidth]{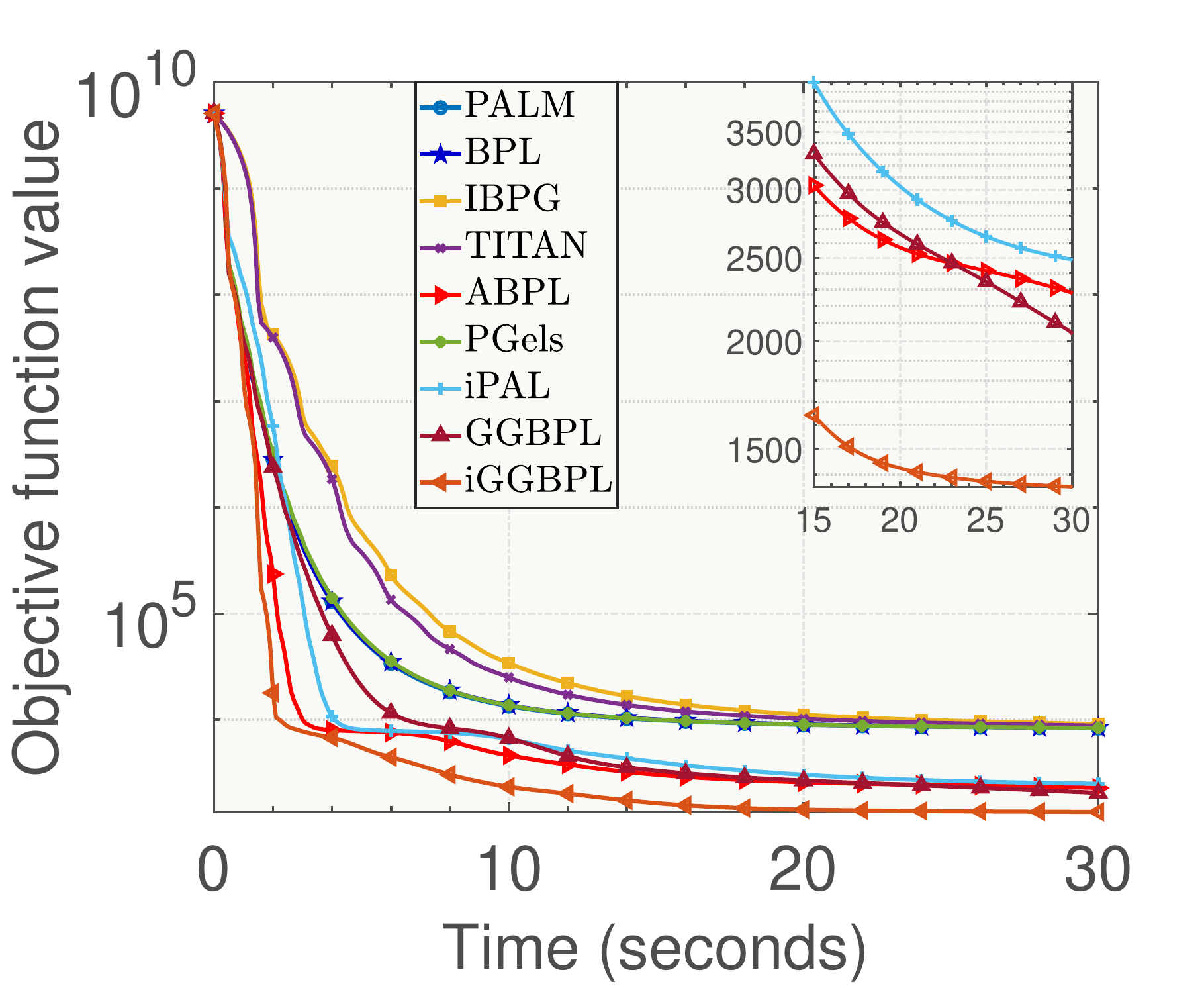}
    \includegraphics[width=0.44\linewidth]{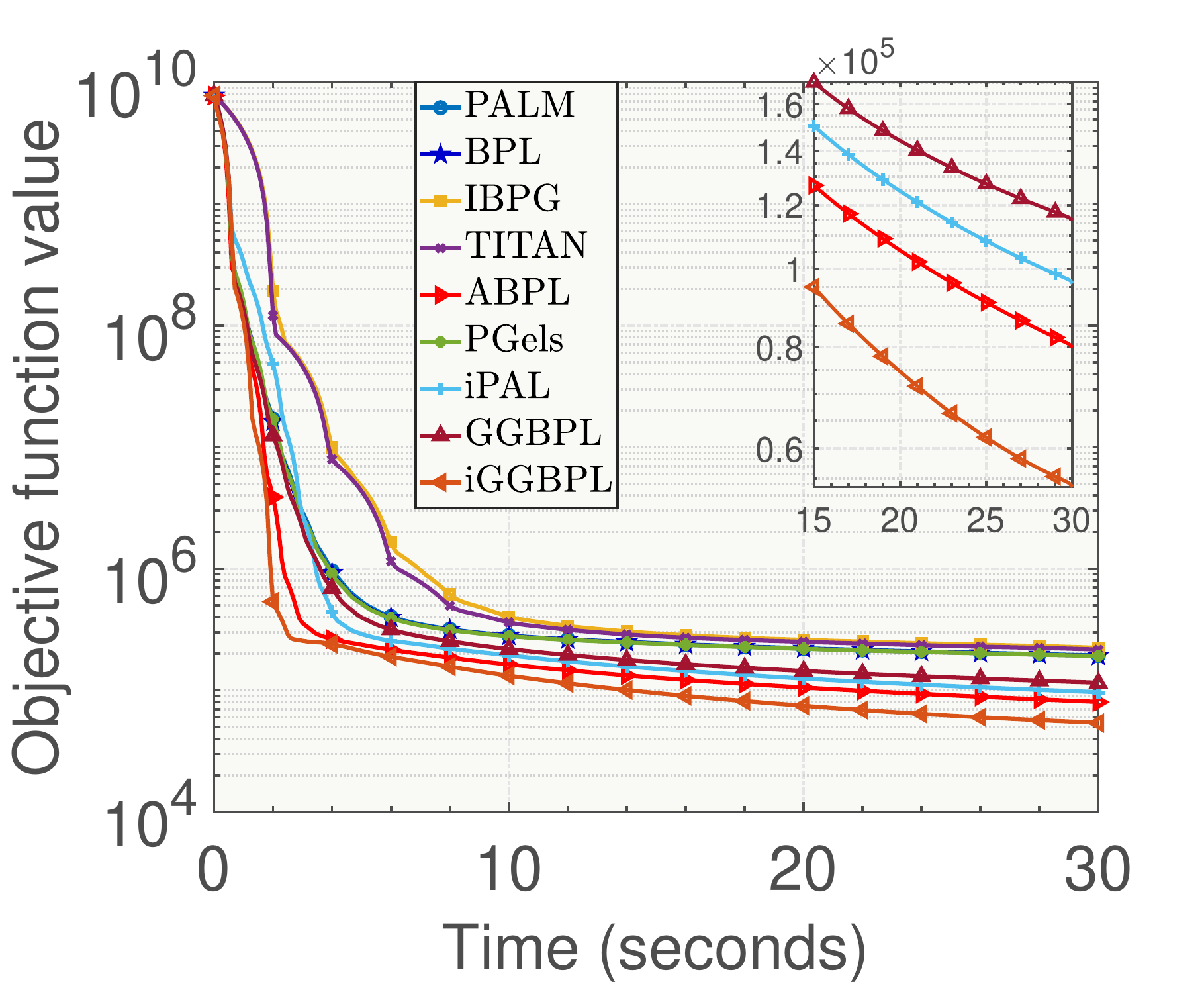}
    \caption{
    Average convergence behavior on the lp\_ship12l (left) and  BASEHOCK (right) datasets with $r=400$. The inset enlarges the final stage objective function values of the four best-performing methods. 
    }\label{figmat1}
\end{figure*}

\begin{figure*}[!h]
    \centering 
    \includegraphics[width=0.44\linewidth]{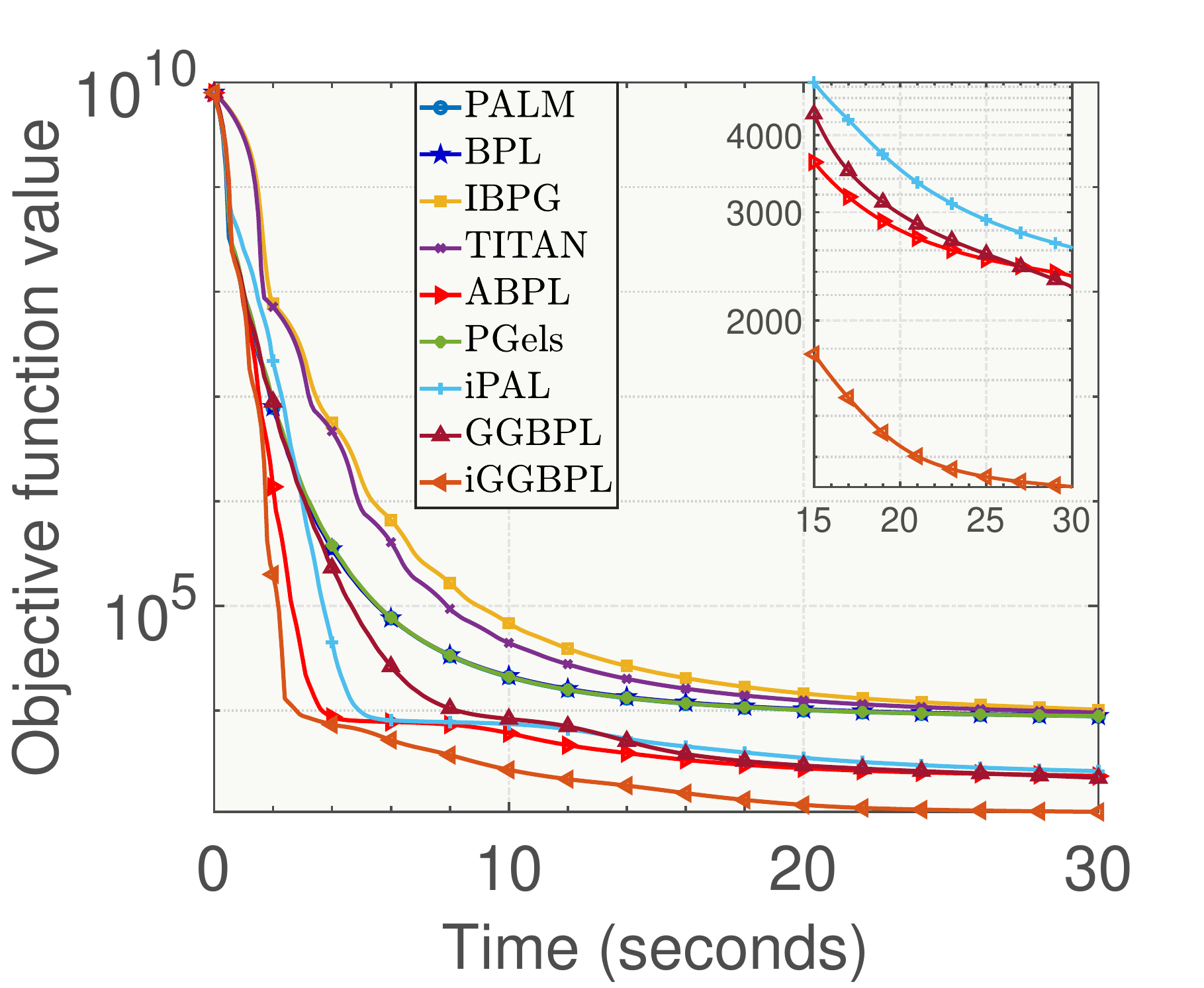}
    \includegraphics[width=0.44\linewidth]{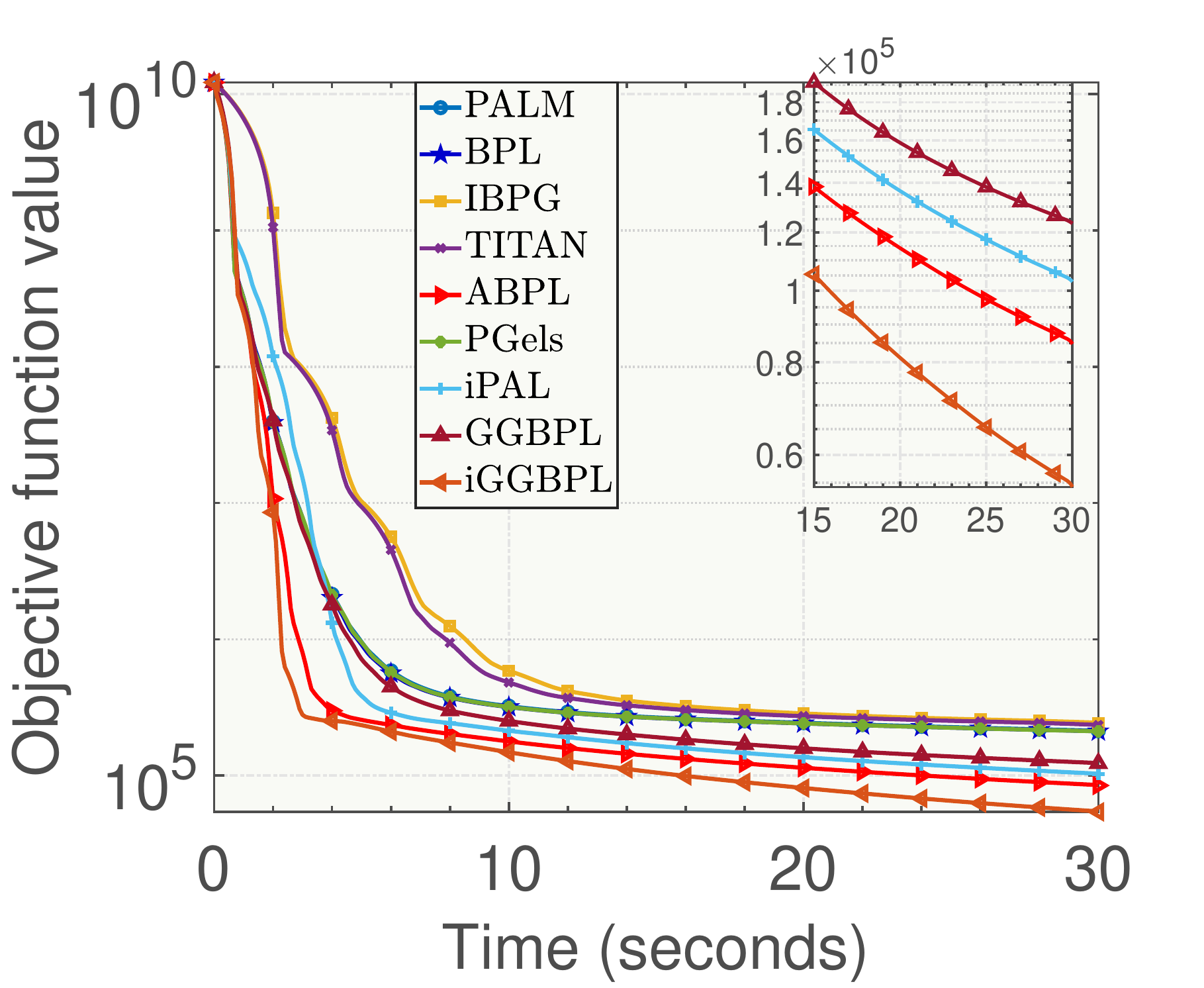}
    \caption{
    Average convergence behavior on the lp\_ship12l (left) and BASEHOCK (right) datasets with $r=500$. The inset enlarges the final stage objective function values of the four best-performing methods. 
    }\label{figmat2}
\end{figure*}


\begin{table*}[!ht]
\centering

\caption{Average and standard deviation results on the matrix datasets. The best performance is highlighted in bold.}
\centering
\label{tabmat1}
\label{tabmat}
\small
\setlength{\tabcolsep}{3pt}
\begin{tabular*}{\linewidth}{|p{1.65cm}|*{3}{>{\centering\arraybackslash}p{\dimexpr(\linewidth-4.65cm-14\tabcolsep-5\arrayrulewidth)/3\relax}>{\centering\arraybackslash}p{1cm}|}}
\hline    
\diagbox[innerwidth=\linewidth]{\multirow{2}{*}{Method}}{Rank} &  \multicolumn{2}{c|}{$r=300$} & \multicolumn{2}{c|}{$r=400$}        &   \multicolumn{2}{c|}{$r=500$}    \\ \hhline{|~|--|--|--|}
   & Rel  & $Wins$   & Rel  & $Wins$ & Rel & $Wins$ \\
\hline

    PALM   & 4.531 $\pm$ 0.031 & 0 & 4.883 $\pm$ 0.040 &0
    & 5.128 $\pm$ 0.046 &0 \\
    BPL   & 4.517 $\pm$ 0.035 & 0 & 4.865 $\pm$ 0.041 &0
    & 5.105 $\pm$ 0.040 &0 \\
    IBPG  & 5.009 $\pm$ 0.051 & 0 & 5.247 $\pm$ 0.07&0
    & 5.468 $\pm$ 0.059 &0 \\
    TITAN  & 4.931 $\pm$ 0.045 & 0 & 5.151 $\pm$ 0.044 &0
    & 5.378 $\pm$ 0.052 &0 \\
    ABPL   & 3.016 $\pm$ 0.020 & 0 & 3.141 $\pm$ 0.028&0
    & 3.233 $\pm$ 0.031 &0 \\
    PGels & 4.503 $\pm$ 0.026 & 0 & 4.858 $\pm$ 0.027 &0
    & 5.105 $\pm$ 0.036 &0 \\
    iPAL  & 3.281 $\pm$ 0.027 & 0 & 3.440 $\pm$ 0.025 &0
    & 3.562 $\pm$ 0.053 &0 \\
    GGBPL & 3.589 $\pm$ 0.037 & 0 & 3.764 $\pm$ 0.03 &0
    & 3.901 $\pm$ 0.037 &0 \\

iGGBPL & $\textbf{2.770}$ $\pmb{\pm}$  $\textbf{0.012}$ & $\textbf{20}$ & $\textbf{2.713}$ $\pmb{\pm}$  $\textbf{0.004}$ & $\textbf{20}$
& $\textbf{2.729}$ $\pmb{\pm}$  $\textbf{0.043}$ & $\textbf{20}$       \\ \hline
  
\end{tabular*}

  \subcaption{\centering Average results by methods on the lp\_ship12l dataset. }

\begin{tabular*}{\linewidth}{|p{1.65cm}|*{3}{>{\centering\arraybackslash}p{\dimexpr(\linewidth-4.65cm-14\tabcolsep-5\arrayrulewidth)/3\relax}>{\centering\arraybackslash}p{1cm}|}}
\hline    
\diagbox[innerwidth=\linewidth]{\multirow{2}{*}{Method}}{Rank} &  \multicolumn{2}{c|}{$r=300$} & \multicolumn{2}{c|}{$r=400$}        &   \multicolumn{2}{c|}{$r=500$}    \\ \hhline{|~|--|--|--|}
   & Rel  & $Wins$   & Rel  & $Wins$ & Rel & $Wins$ \\
\hline

    PALM   & 0.713 $\pm$ 0.005 & 0 & 0.768 $\pm$ 0.006 & 0
    & 0.807 $\pm$ 0.007 & 0\\
    BPL    & 0.711 $\pm$ 0.006 & 0 & 0.766 $\pm$ 0.007 & 0
    & 0.803 $\pm$ 0.006 & 0\\
    IBPG   & 0.788 $\pm$ 0.008 & 0 & 0.826 $\pm$ 0.011 & 0
    & 0.860 $\pm$ 0.009 & 0\\
    TITAN & 0.776 $\pm$ 0.007 & 0 & 0.811 $\pm$ 0.007 & 0
    & 0.846 $\pm$ 0.008 & 0\\
    ABPL   & 0.475 $\pm$ 0.003 & 0 & 0.494 $\pm$ 0.004 & 0
    & 0.509 $\pm$ 0.005 & 0\\
    PGels  & 0.709 $\pm$ 0.004 & 0 & 0.764 $\pm$ 0.004 & 0
    & 0.803 $\pm$ 0.006 & 0\\
    iPAL   & 0.516 $\pm$ 0.004 & 0 & 0.541 $\pm$ 0.004 & 0
    & 0.560 $\pm$ 0.008 & 0\\
    GGBPL  & 0.565 $\pm$ 0.005 & 0 & 0.592 $\pm$ 0.005 & 0
    & 0.614 $\pm$ 0.006& 0\\

iGGBPL & $\textbf{0.436}$ $\pmb{\pm}$  $\textbf{0.002}$ & $\textbf{20}$ & $\textbf{0.427}$ $\pmb{\pm}$  $\textbf{0.004}$ & $\textbf{20}$
& $\textbf{0.429}$ $\pmb{\pm}$  $\textbf{0.007}$ & $\textbf{20}$       \\ \hline
  
\end{tabular*}

  \subcaption{\centering Average results by methods on the BASEHOCK dataset. }

\end{table*}

\subsubsection{Numerical results}
\label{NMFNum}
In this experiment, we test the methods on the dataset from  SuiteSparse\footnote{\url{https://sparse.tamu.edu/}} \cite{10.1145/2049662.2049663}:  lp\_ship12l ($\mathbb{R}^{1151\times5533}$), and we also test the methods on the BASEHOCK\footnote{\url{https://jundongl.github.io/scikit-feature/datasets.html}} dataset ($\mathbb{R}^{1993\times4862}$) \cite{lang1995learning}.

For parameter settings, we set the initial parameters as  
$\gamma^{k}_{i}=1.01, ~\rho_1=\varepsilon=10^{-10}, ~\alpha_1=\alpha_2=1$, 
and the number of non-zero elements in each matrix cannot exceed $30\%$ of the total number of elements. The initial matrix is generated by the uniform distribution. For evaluation metrics, we define relative error (Rel) as Rel$=\frac{\|X-UV\|_{F}}{\|X\|_{F}}$. $Wins$ is defined as the number of times that the corresponding method obtains the lowest relative error among all methods across all independent runs. We adopt these metrics to quantitatively describe the numerical performance of all methods.

We define the maximum running time as $t_{max}$ (s) and set $t_{max}=30$ in this experiment. 
Table \ref{tabmat} reports the average results over 20 independent runs with $r$ varying among $\{300,400,500\}$ on these datasets. Figure \ref{figmat}, Figure \ref{figmat1} and Figure \ref{figmat2} record the evolution of the average objective function value over 20 independent runs with respect to time for $r=300$, $r=400$ and $r=500$, respectively. 
From Table \ref{tabmat}, Figure \ref{figmat}, Figure \ref{figmat1} and Figure \ref{figmat2}, we observe that GGBPL consistently outperforms PALM  and several state-of-the-art inertial methods, while iGGBPL achieves the best overall performance among all compared methods by a clear margin on all datasets.  
These results demonstrate that the proposed generalized geometry operator, constructed using arbitrary inner products and general admissible metrics beyond the standard Euclidean geometry,  enables the block updates to capture the local geometry, scaling, and coupling structures of the corresponding subproblems, while iGGBPL utilizes an inertial term to further enhance the convergence performance of GGBPL, thereby improving the numerical efficiency of the proposed methods.


\subsection{Sparse nonnegative CP decomposition with $\ell_0$-constraints ($\ell_0$-SNCP)}
\subsubsection{$\ell_0$-SNCP model}
Nonnegative CP decomposition (NCP) is widely used to extract latent features from nonnegative multi-way data \cite{chen2022unsupervised}. 
To obtain sparse and interpretable factor matrices, the $\ell_0$-norm constraint can be imposed on the factor matrices since it directly controls the number of nonzero entries. 
However, solving sparse NCP with $\ell_0$-norm constraints ($\ell_0$-SNCP) is also NP-hard, nonsmooth, and nonconvex. 
Additionally, for this problem, the multilinear structure of the CP model provides useful local geometric information of corresponding block variables. 
Therefore, we construct a variable-geometry form of the proposed methods for $\ell_0$-SNCP by introducing nonstandard inner products and metrics derived from the multilinear structure of the $\ell_0$-SNCP model.

The $\ell_0$-SNCP model can be presented as follows. 
\begin{align}
&& &\min\limits_{\left\{A_{i}\right\}_{i=1}^{N}}\frac{1}{2}\|\mathcal{X}-{\llbracket {\left\{A_{i}\right\}_{i=1}^{N}}  
 \rrbracket }\|_{F}^{2}+\sum_{i=1}^N\frac{\alpha_i}{2}\|A_{i}\|_F^2,  \notag \\
&& &s. t.~ A_{i} \geq 0, ~ \|A_{i}\|_{0} \leq s_{i}, 
\label{e43}
\end{align}
where $\mathcal X\in\mathbb R^{d_1\times d_2\times\cdots\times d_N}$ and $A_i\in\mathbb R^{d_i\times R}$, $ \llbracket \ \rrbracket $ represents Kruskal operator, $\odot$ represents the Khatri-Rao product.

\subsubsection{Solving $\ell_0$-SNCP using iGGBPL}
Similarly, if we write Eq. (\ref{e43}) in the form of Eq. (\ref{e11}), then we have 
\begin{align}
    && H(\left\{{A_{i}} \right\}^{N}_{i=1})&=\frac{1}{2}\|\mathcal{X}-{\llbracket {\left\{A_{i}\right\}_{i=1}^{N}} \rrbracket }\|_{F}^{2}+\sum_{i=1}^N\frac{\alpha_i}{2}\|A_i\|_F^2, ~F_{i}(A_{i})=\delta_i(A_{i}).
    \label{SNCPEq}
\end{align}
The definition of $\delta_i(A_{i})$ is the same as that in Eq. (\ref{delta}).

Let $X^{(n)} \in  \mathbb{R}^{d_{n} \times \prod_{i=1, i \neq n}^{N} d_{i}}$ represent the mode-n unfolding of the original tensor $\mathcal{X}$, the mode-n unfolding of the   $ \llbracket {\left\{A_{i}\right\}_{i=1}^{N}} \rrbracket $ can be written as $A_{n}B_{n}^{T}$, where $B_{n}= \prod_{j=1,j\neq n}^{N}\odot A_{j},$.
Therefore, Eq. (\ref{e43}) can be decomposed into some sub-problems, so the $\nabla_{A_{i}}H(\left\{{A_{i}} \right\}^{N}_{i=1})$  is 
\begin{align}
&& \nabla_{A_{i}} H(\left\{{A_{i}} \right\}^{N}_{i=1})&=A_{i}(B_{i})^{T}(B_{i})-X^{(i)}B_{i}+\alpha_iA_i. \label{sncpgrad}
 \end{align}
From the above formulation, it is also easy to see that the local geometric structure of $H(\left\{{A_{i}} \right\}^{N}_{i=1})$ with respect to $A_i$ is characterized by $(B_i)^TB_i+\alpha_iI$, which is induced by the multilinear structure of the CP model and the quadratic regularization term.
Therefore, instead of using the standard inner product and its induced Euclidean norm, we construct the metrics for block variables based on this local geometric information as follows. 
\begin{align}
M_i^k&=\operatorname{Diag}\left((B_i^k)^TB_i^k\mathbf{1}\right)+(\alpha_i+\varepsilon)I, \notag
\end{align}
where $\varepsilon>0$, $\mathbf{1}$ also denotes the all-one vector with a proper dimension. Then the nonstandard inner product and its induced metric are defined as 
\begin{align}
&&&\langle A,B\rangle_i^k=\operatorname{tr}(A M_i^k B^T), \notag\\
&&&d_i^k(A,B)=\|A-B\|_{M_i^k}
=\sqrt{\operatorname{tr}((A-B)M_i^k(A-B)^T)}. \notag
\end{align}
Since $M_i^k$ is derived from the local geometric structure characterized by $(B_i^k)^TB_i^k+\alpha_iI$, the block update scheme of the proposed methods can utilize the local scaling information of the corresponding subproblem of the $\ell_0$-SNCP model.

As a result, Eq. (\ref{SNCPEq}) also satisfies Assumption \ref{assump1}, $d_i^k(A,B)$ also obviously satisfies Assumption \ref{assump2}. By applying the generalized geometry proximal linearized operator (Eq. (\ref{gprox})) to Eq. (\ref{e43}), the update scheme of the proposed methods for solving $\ell_0$-SNCP can be expressed as follows.  
\begin{align}
 && A_{i}^{k+1} &\in Gprox_{\sigma_{i}^{k} F_{i}}^{d_i^k}
 \left(y_i^k;\nabla_{A_i} h_i(y_i^k)(M_i^k)^{-1}\right) \notag \\
 && &= \operatorname*{\arg\min} \limits_{A}
 \left\{\|A-U^{k}_{i}\|_{M_i^k}^{2}: A\geq 0,~\|A\|_{0}\le s_{i}\right\},  \notag
\end{align}
where $\sigma_{i}^{k}=\frac{1}{\gamma_{i}^{k}},~ \gamma^{k}_{i}>1$ and 
\begin{align}
&& & U^{k}_{i}=y_i^k-\sigma_{i}^{k}\nabla_{A_i} h_i(y_i^k)(M_i^k)^{-1},\notag
\end{align}


\subsubsection{Numerical results}
We test the methods on the BreastMNIST\footnote{\url{https://github.com/MedMNIST/MedMNIST}}, and microPNW\footnote{\url{https://github.com/niyiyu/PNW-ML}} datasets. 
The BreastMNIST dataset is a three-dimensional tensor  ($\mathbb{R}^{780\times224\times224}$).  
For the microPNW dataset, each sample in the microPNW dataset is transformed into the form of time frames $\times$ frequency bins $\times$ channels, thus the microPNW dataset is a four-dimensional tensor ($\mathbb{R}^{100\times150\times3\times50}$), and we also normalize the value of each feature in the microPNW dataset to the range of 0 to 1.

We set $\alpha_i=1~(\forall i\in [N])$, all other parameter settings are the same as those in the previous experiment. 
We also define relative error (Rel) as Rel$=\frac{\|\mathcal{X}-{\llbracket  \{ A_{i}\}_{i=1}^N  \rrbracket }\|_{F}}{\|\mathcal{X}\|_{F}}$. 
The definition of $Wins$ is the same as that in the previous experiment.

We set $t_{max}=40$ in this experiment. 
Table \ref{tabten} reports the average results over 20 independent runs with $R$ varying among $\{50,60,70\}$ on these datasets. Figure \ref{figten}, Figure \ref{figten1} and Figure \ref{figten2} record the evolution of the average objective function value over 20 independent runs with respect to time for $R=50$, $R=60$ and $R=70$, respectively. 
From Table \ref{tabten}, Figure \ref{figten}, Figure \ref{figten1} and Figure \ref{figten2}, GGBPL again outperforms PALM and several state-of-the-art inertial methods, while iGGBPL also achieves the best overall performance among all compared methods by a clear margin on all datasets. Together with the results for $\ell_0$-SNMF, these experimental results demonstrate that the proposed generalized geometry operator is not restricted to matrix factorization but remains numerically effective for higher-order tensor decomposition, where the corresponding block subproblems exhibit more complex local geometry, scaling, and coupling structures, thereby further confirming the effectiveness of the proposed methods. 
Combined with the results for $\ell_0$-SNMF, these results demonstrate that the numerical advantages of the proposed generalized geometry operator are not limited to two-block matrix factorization problems, but extend to higher-order multiblock tensor decomposition problems with more complex multilinear coupling structures. This further demonstrates the adaptability and scalability of the proposed method across different problem structures and CP ranks.

\begin{table*}[!ht]

\centering

\caption{Average and standard deviation results on the tensor datasets. The best performance is highlighted in bold.}
\centering
\label{tabten1}
\label{tabten}
\small
\setlength{\tabcolsep}{3pt}
\begin{tabular*}{\linewidth}{|p{1.65cm}|*{3}{>{\centering\arraybackslash}p{\dimexpr(\linewidth-4.65cm-14\tabcolsep-5\arrayrulewidth)/3\relax}>{\centering\arraybackslash}p{1cm}|}}
\hline    
\diagbox[innerwidth=\linewidth]{\multirow{2}{*}{Method}}{Rank} &  \multicolumn{2}{c|}{$R=50$} & \multicolumn{2}{c|}{$R=60$}     &   \multicolumn{2}{c|}{$R=70$}    \\ \hhline{|~|--|--|--|}
   & Rel  & $Wins$   & Rel  & $Wins$ & Rel & $Wins$ \\
\hline

        PALM   & 0.316 $\pm$ 0.006 & 0 & 0.311 $\pm$ 0.004 &0
        & 0.312 $\pm$ 0.005 &0\\
    BPL    & 0.313 $\pm$ 0.012 & 0 & 0.310 $\pm$ 0.010 &0
    & 0.313 $\pm$ 0.007 &0\\
    IBPG   & 0.301 $\pm$ 0.011 & 0 & 0.297 $\pm$ 0.009 &0
    & 0.296 $\pm$ 0.010 &0\\
    TITAN  & 0.298 $\pm$ 0.012 & 0 & 0.294 $\pm$ 0.010 &0
    & 0.292 $\pm$ 0.007 &0\\
    ABPL   & 0.290 $\pm$ 0.011 & 0 & 0.286 $\pm$ 0.005 &0
    & 0.284 $\pm$ 0.011 &0\\
    PGels  & 0.315 $\pm$ 0.006 & 0 & 0.310 $\pm$ 0.004 &0
    & 0.311 $\pm$ 0.005 &0\\
    iPAL   & 0.279 $\pm$ 0.003 & 0 & 0.273 $\pm$ 0.002 &0
    & 0.269 $\pm$ 0.002 &0\\
    GGBPL & 0.286 $\pm$ 0.003 & 0 & 0.280 $\pm$ 0.003 &0
    & 0.277 $\pm$ 0.002 &0\\

iGGBPL & $\textbf{0.253}$ $\pmb{\pm}$  $\textbf{0.001}$ & $\textbf{20}$ & $\textbf{0.246}$ $\pmb{\pm}$  $\textbf{0.001}$ & $\textbf{20}$
& $\textbf{0.240}$ $\pmb{\pm}$  $\textbf{0.001}$ & $\textbf{20}$   \\ \hline

\end{tabular*}
  \subcaption{\centering Average results by methods on the BreastMNIST dataset. }

\begin{tabular*}{\linewidth}{|p{1.65cm}|*{3}{>{\centering\arraybackslash}p{\dimexpr(\linewidth-4.65cm-14\tabcolsep-5\arrayrulewidth)/3\relax}>{\centering\arraybackslash}p{1cm}|}}
\hline    
\diagbox[innerwidth=\linewidth]{\multirow{2}{*}{Method}}{Rank} &  \multicolumn{2}{c|}{$R=50$} & \multicolumn{2}{c|}{$R=60$}     &   \multicolumn{2}{c|}{$R=70$}    \\ \hhline{|~|--|--|--|}
   & Rel  & $Wins$  & Rel  & $Wins$ & Rel & $Wins$ \\
\hline

    PALM   & 0.107 $\pm$ 0.004 & 0 & 0.108 $\pm$ 0.003 & 0
    & 0.107 $\pm$ 0.006 & 0\\
    BPL    & 0.111 $\pm$ 0.006 & 0 & 0.109 $\pm$ 0.005 & 0
    & 0.107 $\pm$ 0.007 & 0\\
    IBPG   & 0.104 $\pm$ 0.005 & 0 & 0.103 $\pm$ 0.004 & 0
    & 0.102 $\pm$ 0.007 & 0\\
    TITAN  & 0.103 $\pm$ 0.004 & 0 & 0.101 $\pm$ 0.003& 0
    &0.101 $\pm$ 0.005 & 0\\
    ABPL   & 0.095 $\pm$ 0.001 & 0 & 0.092 $\pm$ 0.001 & 0
    & 0.089 $\pm$ 0.001 & 0\\
    PGels  & 0.107 $\pm$ 0.004 & 0 & 0.107 $\pm$ 0.004 & 0
    & 0.107 $\pm$ 0.006 & 0\\
    iPAL   & 0.096 $\pm$ 0.001 & 0 & 0.092 $\pm$ 0.001 & 0
    & 0.089 $\pm$ 0.001 & 0\\
    GGBPL  & 0.095 $\pm$ 0.001 & 0 & 0.092 $\pm$ 0.002 & 0
    & 0.090 $\pm$ 0.002 & 0\\

iGGBPL & $\textbf{0.086}$ $\pmb{\pm}$  $\textbf{0.001}$ & $\textbf{20}$ & $\textbf{0.086}$ $\pmb{\pm}$  $\textbf{0.001}$ & $\textbf{20}$
& $\textbf{0.083}$ $\pmb{\pm}$  $\textbf{0.004}$ & $\textbf{20}$   \\ \hline

\end{tabular*}
  \subcaption{\centering Average results by methods on the microPNW dataset. }



\end{table*}


\begin{figure}[!h]
    \centering 
    \includegraphics[width=0.48\linewidth]{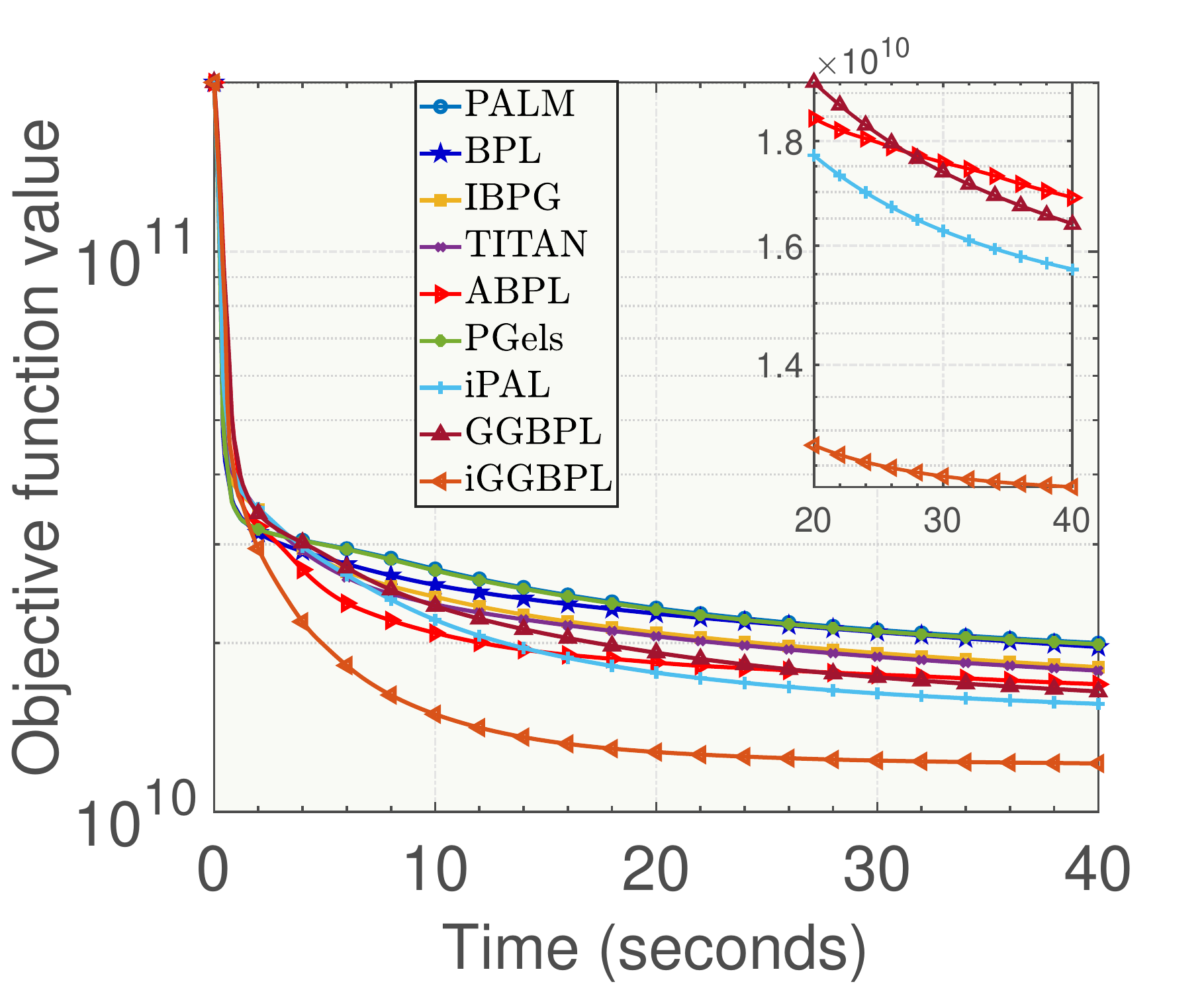}
    \includegraphics[width=0.48\linewidth]{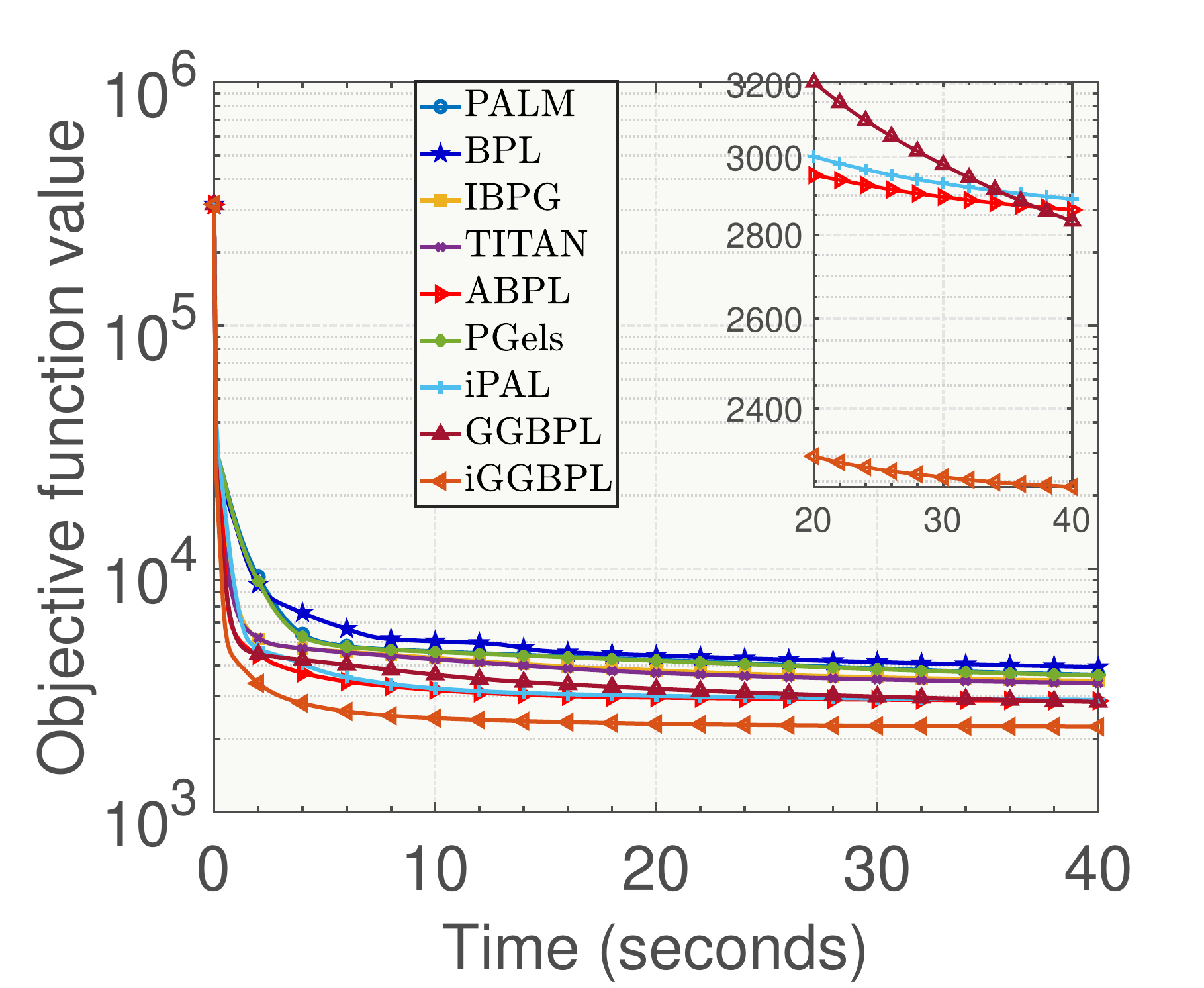}
    
    \caption{
    Average convergence behavior on the BreastMNIST (left) and microPNW (right) datasets. The inset enlarges the final stage objective function values of the four best-performing methods.  
    }\label{figten}
    
\end{figure}

\begin{figure*}[!h]
    \centering 
    \includegraphics[width=0.44\linewidth]{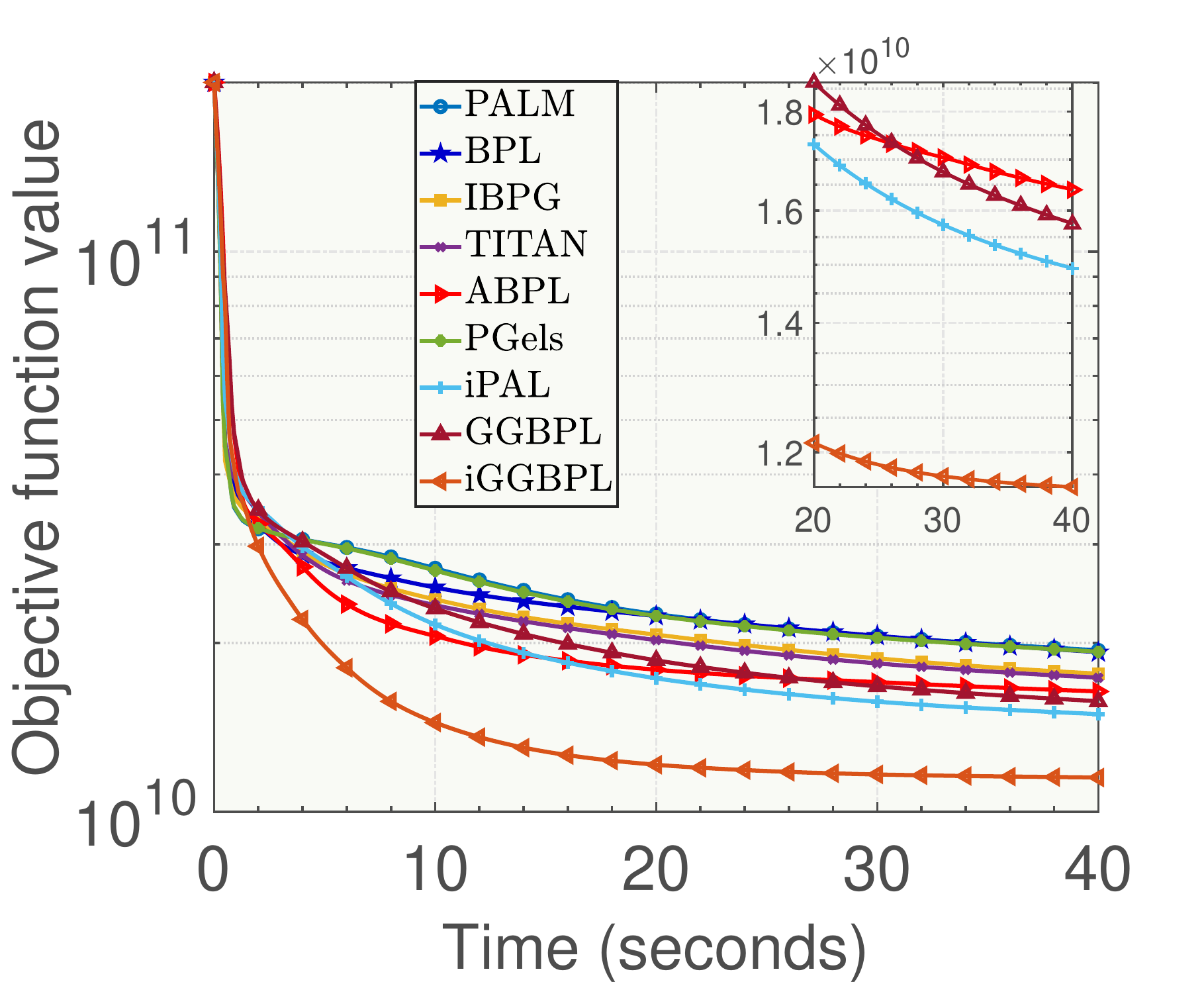}
    \includegraphics[width=0.44\linewidth]{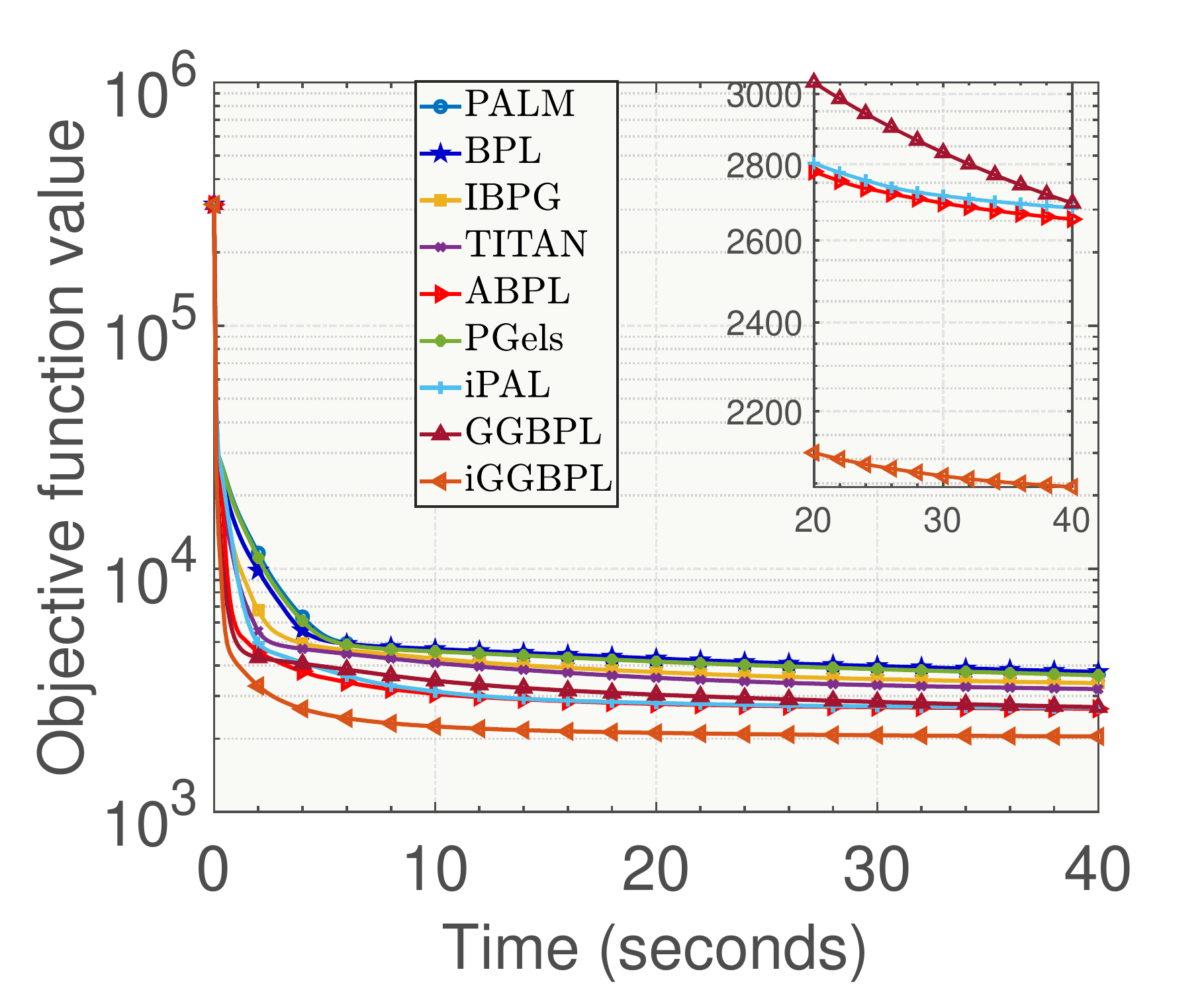}
    \caption{
    Average convergence behavior on the BreastMNIST (left) and microPNW (right) datasets with $R=60$. The inset enlarges the final stage objective function values of the four best-performing methods. 
    }\label{figten1}

\end{figure*}
\begin{figure*}[!h]
    \centering 
    \includegraphics[width=0.44\linewidth]{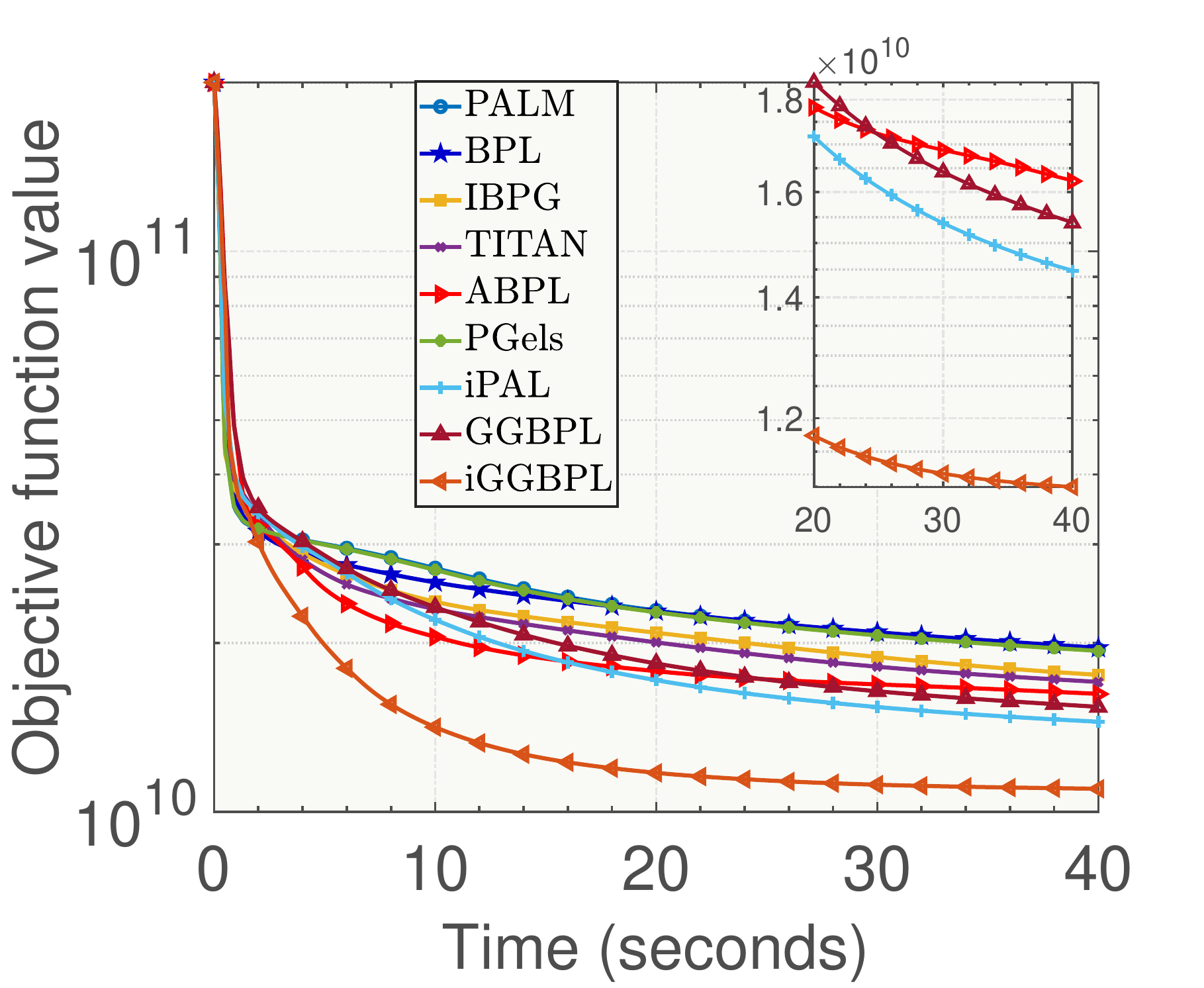}
    \includegraphics[width=0.44\linewidth]{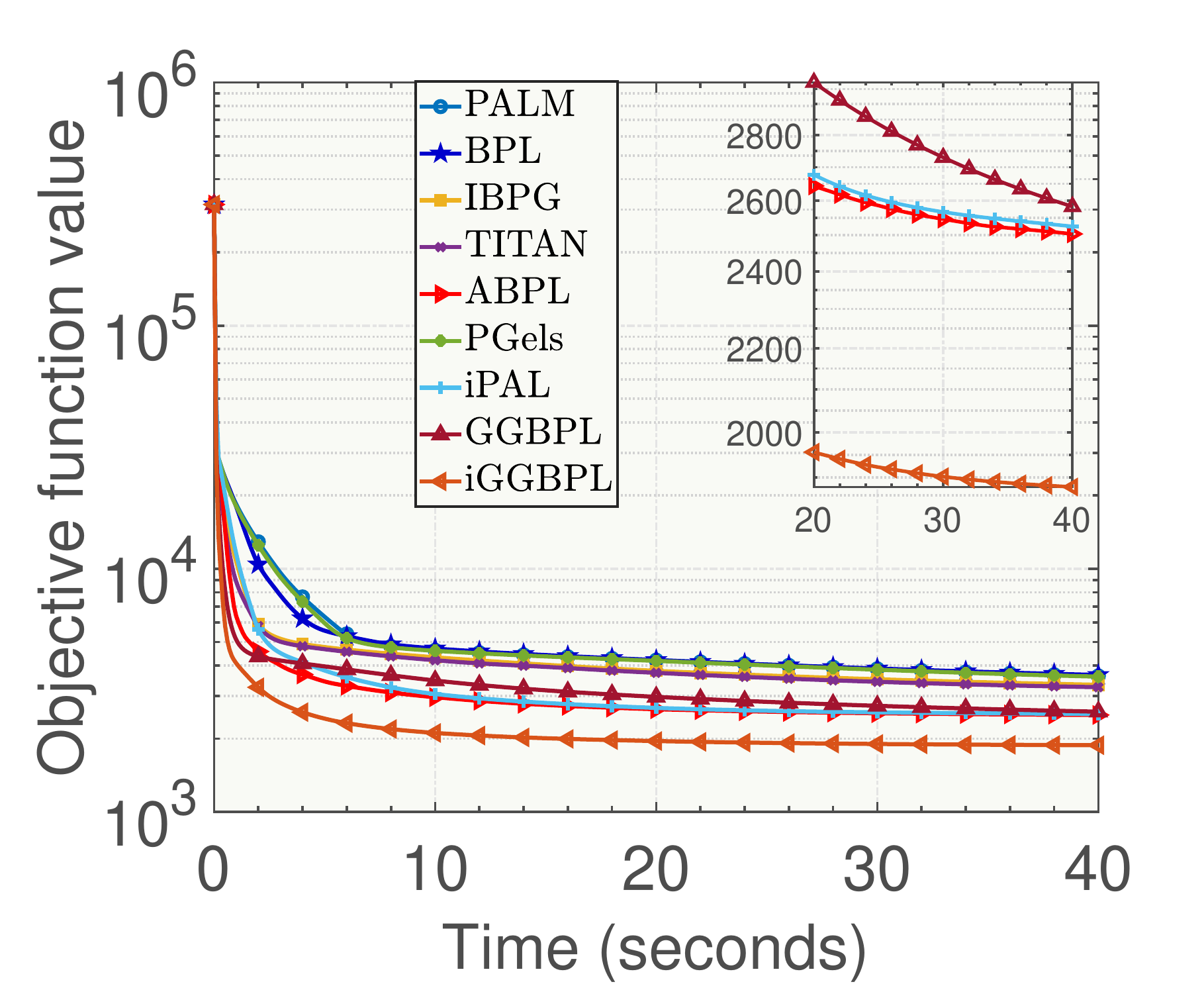}
    \caption{
    Average convergence behavior on the BreastMNIST (left) and microPNW (right) datasets with $R=70$. The inset enlarges the final stage objective function values of the four best-performing methods. 
    }\label{figten2}

\end{figure*}

\section{Conclusion}
\label{conclu}

In this paper, we proposed a generalized geometry proximal linearized operator and developed the Generalized Geometry Block Proximal Linearized (GGBPL) method based on this operator. 
The proposed operator allows the block surrogate functions to be constructed using arbitrary inner products and general admissible metrics, rather than being restricted to the standard Euclidean geometry. We also introduced the inertial version of GGBPL named iGGBPL to accelerate convergence. 
This design enables the block updates of the proposed methods to utilize the local geometric information of various target problems and yields practical convergent schemes for directly solving these problems, thereby improving the flexibility and applicability of the proposed methods. 
We further established a unified convergence framework under this generalized geometry, within which we proved that our methods guarantee convergence for this class of problems, established the global convergence of the generated sequence to a critical point, and derived the convergence rate. We also established an $\mathcal{O}(\varepsilon^{-2})$ iteration complexity bound for obtaining an $\varepsilon$-stationary point, providing a finite-iteration guarantee under this generalized geometry. 
We applied our proposed methods to solve two nonconvex and nonsmooth problems: $\ell_0$-SNMF and $\ell_0$-SNCP. 
Numerical results demonstrated the superior numerical performance of our proposed methods over several state-of-the-art methods.

\section*{Statements and Declarations}

\bmhead{Competing interests}
The author declares no competing interests. 

\bmhead{Data availability}
The datasets used in this study are publicly available: lp\_ship12l from the SuiteSparse Matrix Collection (\url{https://sparse.tamu.edu/}), BASEHOCK from the scikit-feature dataset repository (\url{https://jundongl.github.io/scikit-feature/datasets.html}), BreastMNIST from MedMNIST (\url{https://github.com/MedMNIST/MedMNIST}), and microPNW from PNW-ML (\url{https://github.com/niyiyu/PNW-ML}).

\bibliography{references}

\end{document}